\documentclass[a4paper,12pt]{amsart}
\usepackage{amsmath,amsfonts,amssymb,amsthm}
\usepackage{mathtools,enumitem,cite,bbm}
\usepackage{graphics,fullpage}
\usepackage{tikz}
\allowdisplaybreaks

\DeclarePairedDelimiter\floor{\lfloor}{\rfloor}

\newtheorem{theorem}{Theorem}[section]
\newtheorem{lemma}[theorem]{Lemma}
\newtheorem{proposition}[theorem]{Proposition}
\newtheorem{corollary}[theorem]{Corollary}
\newtheorem{assumption}{Assumption}

\theoremstyle{remark}
\newtheorem{remark}[theorem]{Remark}

\begin{document}

\title[Quenched localisation in the BTM with slowly varying traps]{Quenched localisation in the Bouchaud trap \\ model with slowly varying traps}
\author{David A. Croydon$^1$}
\address{$^1$Department of Statistics, University of Warwick}
\email{d.a.croydon@warwick.ac.uk}
\author{Stephen Muirhead$^2$}
\address{$^2$Department of Mathematics, University College London (current address: Mathematical Institute, University of Oxford)}
\email{muirhead@maths.ox.ac.uk}
\thanks{Part of this article was written whilst the first author was a Visiting Associate Professor at Kyoto University, Research Institute for Mathematical Sciences. He would like to thank Takashi Kumagai and Ryoki Fukushima for their kind and generous hospitality. The second author was partially supported by a Graduate Research Scholarship from University College London and the Leverhulme Research Grant RPG-2012-608 held by Nadia Sidorova, and partially supported by the Engineering \& Physical Sciences Research Council (EPSRC) Fellowship EP/M002896/1 held by Dmitry Belyaev. We would like to thank an anonymous referee for their thoughtful comments.}
\date{\today}

\begin{abstract}
We consider the quenched localisation of the Bouchaud trap model on the positive integers in the case that the trap distribution has a slowly varying tail at infinity. Our main result is that for each $N \in \{2, 3, \ldots\}$ there exists a slowly varying tail such that quenched localisation occurs on exactly $N$ sites. As far as we are aware, this is the first example of a model in which the exact number of localisation sites are able to be `tuned' according to the model parameters. Key intuition for this result is provided by an observation about the sum-max ratio for sequences of independent and identically distributed random variables with a slowly varying distributional tail, which is of independent interest.
\end{abstract}
\maketitle

\section{Introduction}

This article studies localisation properties of the Bouchaud trap model (BTM) on the positive integers in the case of slowly varying traps. To define the BTM, we first introduce a trapping landscape $\sigma=(\sigma_x)_{x \in \mathbb{Z}^+}$, which is a collection of independent and identically distributed (i.i.d.)\ strictly-positive random variables built on a probability space with probability measure~$\mathbf{P}$. Conditional on $\sigma$, the dynamics of the BTM are given by a continuous-time $\mathbb{Z}^+$-valued Markov chain $X=(X_t)_{t\geq 0}$, started from the origin, with transition rates
\begin{equation}\label{trates}
 w_{x \to y} = \begin{cases}
 \frac{1}{2 \sigma_x}, & \text{if }  y \sim x, \\
 0, & \text{otherwise,}
\end{cases}
\end{equation}
where $y \sim x$ means that $x$ and $y$ are nearest neighbours in $\mathbb{Z}^+$. We denote the law of $X$ conditional on $\sigma$, the so-called `quenched' law of the BTM, by $P_\sigma$. Our focus is on the case in which the trap distribution $\sigma_0$ has a slowly varying tail at infinity, i.e.\ when the (non-decreasing, unbounded, c\`{a}dl\`{a}g) function
\[ L(u) := \frac{1}{\mathbf{P}(\sigma_0 > u)} \]
satisfies the \text{slow-variation} assumption
\begin{equation}\label{eq:slow}
 \lim_{u \to \infty} \frac{L(u v)}{L(u)} = 1 \quad \text{for all} \ v > 0.
 \end{equation}

Slowly varying trap models arise naturally in the study of certain random walks in random media, such as biased random walks on critical Galton-Watson trees \cite{Croydon13}, and spin-glass dynamics on subexponential time scales \cite{BenArous12, Bovier13}. They also have parallels with Sinai's random walk \cite{Sinai82}, as reflected in the logarithmic rate of escape to infinity and strong localisation properties of that model. With regards to the BTM with slowly varying traps in particular, recent work has studied localisation \cite{Muirhead15}, ageing \cite{Gun13}, and scaling limits \cite{Croydon15}, which are qualitatively different from the equivalent phenomena in the case of integrable or regularly varying traps.

In this work we continue the study of the BTM with slowly varying traps by considering more delicate localisation properties of the model, namely those that hold under the quenched law for typical realisations of the trapping landscape. We expect that similar quenched localisation properties hold throughout the class of general slowly varying trap models. For simplicity we have chosen to work in the one-sided case (i.e.\ on the positive integers, rather than on the integers); this avoids some of the technical difficulties present in the two-side case, yet still exhibits the phenomena that interest us. We make some remarks about quenched localisation in the BTM on the integers below.

\subsection{Localisation in the BTM}

It was recently shown in \cite{Muirhead15} that the BTM on the integers with slowly varying traps exhibits two-site localisation in probability, that is, there exists a  ($\mathbf{P}$-measurable) set-valued process $\Gamma_t$ such that $|\Gamma_t|=2$ and, as $t \to \infty$,
\begin{align}
\label{eq:twosite}
 P_\sigma(X_t \in \Gamma_t) \to 1  \qquad \text{in } \mathbf{P} \text{-probability,}
 \end{align}
and, moreover, no set-valued process $\Gamma_t$ with $|\Gamma_t| = 1$ satisfies equation \eqref{eq:twosite}; note that here the probability measures $\mathbf{P}$ and $P_\sigma$ refer to the BTM on the integers. The basic fact underlying this localisation result is that the cumulative sum of i.i.d.\ sequences of slowly varying random variables is asymptotically dominated, with high probability, by the maximal term. Translated to the setting of the BTM, this property implies that the BTM with slowly varying traps is highly likely to be located, at any sufficient large time, on the largest traps on the positive and negative half-lines that are `within reach' of the BTM by this time. Note that two-site localisation is not exhibited in the BTM with regularly varying or integrable traps.

We seek to establish an almost-sure analogue of the above result, that is, to determine the smallest ($\mathbf{P}$-measurable) set-valued process $\Gamma_t$ such that, as $t \to \infty$,
\[  P_\sigma(X_t \in \Gamma_t) \to 1 \qquad \mathbf{P}\text{-almost-surely};\]
whenever we refer to `quenched localisation', it is a limit such as this that we mean. In the one-sided case, it is possible to check that the analogue to (\ref{eq:twosite}) holds with $|\Gamma_t|=1$ (see Theorem \ref{thm:maincl} below). From this, simple heuristics suggest that the strongest form of quenched localisation that one might hope to observe is that there exists a localisation set $\Gamma_t$ such that $|\Gamma_t| = 2$; this follows from the fact that the probability mass function eventually moves out to infinity and so, by continuity, must be spread over at least two sites at arbitrarily large times.

In this paper we prove that this strongest form of quenched localisation is actually attained for certain examples of the BTM with slowly varying traps; in other words, we prove that there exists a certain class of slowly varying tails for which the BTM localises on two sites eventually almost-surely, and for which one localisation site would be insufficient.

More surprisingly perhaps, for each $N \in \{2, 3, \ldots \}$ we show that there exists a slowly varying tail such that quenched localisation occurs on exactly $N$ sites. As such, the BTM with slowly varying traps is an example of a model that exhibits quenched localisation on a finite number of sites, with the exact number of localisation sites able to be tuned by adjusting the parameters of the model. As far as we are aware, this is the first known example of such a model.

\subsection{Main results}

In this section, we describe our main results on localisation in the BTM on the positive integers. We assume throughout that the trap distribution satisfies the slow-variation assumption (\ref{eq:slow}). As a preliminary, we first state the one-sided analogue of equation \eqref{eq:twosite}; that is, we establish the complete localisation in probability of the BTM on the positive integers.

\begin{theorem}[Complete localisation in probability]
\label{thm:maincl}
It is possible to define a ($\mathbf{P}$-measurable) process $Z_t$ such that, as $t \to \infty$,
\begin{align}
   P_\sigma(X_t = Z_t) \to 1 \quad \text{in } \mathbf{P}\text{-probability.}\nonumber
\end{align}
\end{theorem}

The main focus of the paper is to establish quenched analogues of Theorem \ref{thm:maincl}. Interestingly, quenched localisation in the BTM turns out to depend on rather fine properties of the trap distribution $\sigma_0$. To introduce these properties, we first recall that a function $L$ is said to be \textit{second-order slowly varying with rate} $g$, if there exist functions $g, k$ such that $g(u) \to 0$ as $u \to \infty$ and
\begin{align}\label{sosv}
\lim_{u \to \infty} \frac{\frac{L(u v)}{L(u)} - 1}{ g(u)}  =  k(v), \quad \text{for any } v > 0,
\end{align}
and where there exists a $v$ such that $k(v)\neq 0$ and $k(uv)\neq k(u)$ for all $u>0$; as discussed in \cite{Croydon15}, in our setting it follows from this assumption that $g$ itself is slowly varying and (possibly after multiplying by a constant) we can take $k(v)=\log v$. Second-order slow-variation is a natural strengthening of the slow-variation property \eqref{eq:slow}, giving more precise information about the fluctuations of $L$ at infinity; see \cite[Chapter 3]{Bingham87} for an overview of second-order slow-variation.

In particular, each of our main results will depend on the assumption that $L$ is second-order slowly varying. In addition, it will be convenient to assume certain extra regularity condition on $L$ and $g$, namely that $L$ is continuous (which avoids complications in our treatment of records of the sequence $\sigma$) and that the decay of $g$ is eventually monotone (which allow us to control the decay of $g(x)$ through bounds on $x$).

\begin{assumption}[Second-order slow-variation assumption]
\label{assumpt:sosv} The function $L$ is continuous, and satisfies \eqref{sosv} with second-order slow-variation rate $g$ that is eventually monotone decreasing.
\end{assumption}

We shall now explain how the quenched localisation behaviour of the BTM depends on the precise asymptotic decay of the second-order slow-variation rate $g$. For the rest of this section we work under Assumption \ref{assumpt:sosv}, and abbreviate the function
\[ d(u) := g(L^{-1}(u)) , \]
where $L^{-1}$ denotes the right-continuous inverse of $L$, noting that $d(u) \to 0$ as $u \to \infty$. Define the integer
\begin{align}
\label{eq:N}
 N := \min \left\{\ell \in \{2, 3, \ldots\}  :  \sum_{n \in \mathbb{N}} \left(d(e^n) \log n\right)^{\ell-1} < \infty    \right\} ,
\end{align}
setting $N = \infty$ if no such $\ell$ exists. Our first main theorem (Theorem \ref{thm:main1}) identifies $N$ as the number of quenched localisation sites of the BTM. Before stating this result, we first need to introduce an additional assumption that is necessary for certain aspects of our results to hold. This assumption acts to exclude boundary cases, in which the number of localisation sites in a sense falls intermediate between two integers.

\begin{assumption}[Exclusion of boundary cases]
\label{assumpt:g22}  It is the case that $N < \infty$, and \\
\indent
\[ \text{(a)} \quad \sum_{n \in \mathbb{N}} d(e^n)^{N-2} = \infty \quad,  \qquad \text{(b)}  \quad \sum_{n \in \mathbb{N}} d(e^n)^{N-1} (\log n)^N < \infty . \]
\end{assumption}

\begin{remark}
We note that Assumptions \ref{assumpt:sosv} and \ref{assumpt:g22} are satisfied for a wide range of slowly varying distributions~$\sigma_0$. The main examples we have in mind are distributions satisfying

\begin{align}
\label{eq:L}
L(u) := \exp \{ (\log (1+ u))^\gamma \}  , \quad \gamma \in (0, 1) ,
\end{align}
for which $g(u) =\gamma (\log (1+ u))^{\gamma-1}$, $d(e^n)\sim\gamma n^{-\frac{1-\gamma}{\gamma}}$, and
\[ N = 2 + \left \lfloor \frac{\gamma}{1-\gamma} \right \rfloor  .\]
In this example, we observe that $N = 2$ if and only if $\gamma < 1/2$, that $N \to \infty$ as $\gamma \to 1$, and moreover that any $N \in \{2, 3, \ldots\}$ is attainable by selecting an appropriate $\gamma \in (0, 1)$. Other classes of slowly varying distribution for which our results hold are those with logarithmic decay ($L(u) = (1+\log (1+ u))^\gamma, \gamma > 0$), or double logarithmic decay ($L(u) = (1+\log (1 + \log (1+ u)))^\gamma, \gamma > 0$); in both cases $L$ satisfies  Assumptions \ref{assumpt:sosv} and~\ref{assumpt:g22} with $N = 2$.
\end{remark}

Our first main result establishes the property of $N$-site localisation almost-surely.

\begin{theorem}[$N$-site localisation]
\label{thm:main1} If Assumption \ref{assumpt:sosv} holds, then there exists a ($\mathbf{P}$-measurable) set-valued process $\Gamma_t$ satisfying $|\Gamma_t| \le N$ such that, as $t \to \infty$,
\begin{align}
\label{eq:asloc}
   P_\sigma(X_t \in \Gamma_t) \to 1 \quad \mathbf{P}\text{-almost-surely.}
\end{align}
Moreover if Assumption \ref{assumpt:g22} also holds, then there is no set-valued process $\Gamma_t$ satisfying $|\Gamma_t| < N$ such that \eqref{eq:asloc} holds.
\end{theorem}

The first claim of Theorem \ref{thm:main1} states that the probability mass of the BTM is asymptotically supported by a certain ($\mathbf{P}$-measurable) collection of $N$ sites, where $N \ge 2$. To deduce the second claim that no smaller set will do, we study the most favoured site of the BTM. In particular, our second main result (Theorem \ref{thm:main2}) shows that the associated probability mass fluctuates infinitely often between the bounds of $1/N$ and $1$; Figure \ref{trajpic} shows a sketch of a typical trajectory of this probability mass. An immediate implication of this is that there exist arbitrary large times such that the probability mass of the BTM is approximately uniform across each of the $N$ localisation sites, which is sufficient to complete the proof of Theorem \ref{thm:main1}.

\begin{theorem}[Probability mass on most favoured site]
\label{thm:main2} If Assumption \ref{assumpt:sosv} holds, then $\mathbf{P}$-almost-surely,
\[ \liminf_{t \to \infty} \, \sup_{x \in \mathbb{Z}^+} P_\sigma(X_t = x) \ge 1/N \quad \text{and} \quad  \limsup_{t \to \infty} \, \sup_{x \in \mathbb{Z}^+}P_\sigma(X_t = x)  = 1  .\]
Moreover if Assumption \ref{assumpt:g22} also holds, then $\mathbf{P}$-almost-surely,
 \[ \liminf_{t \to \infty} \, \sup_{x \in \mathbb{Z}^+} P_\sigma(X_t = x) = 1/N . \]
\end{theorem}

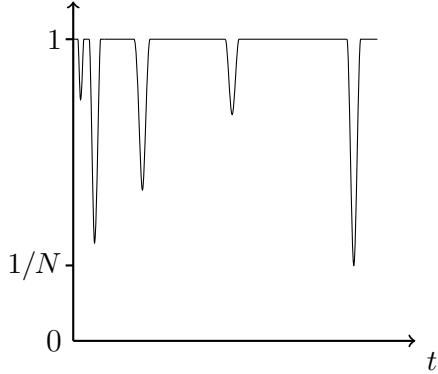
\begin{figure}[ht]
\begin{tikzpicture}
\draw[thick,->] (0,0) -- (4.5,0) node[anchor=north west] {\small $t$};
\draw[thick,->] (0,0)  node[anchor=east] {$0$} -- (0,1)  node[anchor=east] {\small $1/N$} -- (0,4) node[anchor=east] {\small $1$} -- (0,4.5) node[anchor=south] {};
\draw[thick] (0,1) -- (-0.1,1);
\draw[thick] (0,4) -- (-0.1,4);
\draw (0,4) -- (0.06,4) .. controls (0.072,4) and (0.084,3.2) .. (0.096,3.2) -- (0.096,3.2) .. controls (0.116,3.2) and (0.128,4) .. (0.14,4) -- (0.21,4);
\draw (0.21,4) -- (0.21,4) .. controls (0.23,4) and (0.26,1.3) .. (0.28,1.3) -- (0.28,1.3) .. controls (0.31,1.3) and (0.34,4) .. (0.36,4) -- (0.41,4);
\draw (0.41,4) -- (0.8,4) .. controls (0.85,4) and (0.88,2) .. (0.91,2) -- (0.91,2) .. controls (0.94,2) and (0.97,4) .. (1.01,4) -- (2,4);
\draw (2,4) .. controls (2.03,4) and (2.06,3) .. (2.09,3) -- (2.09,3) .. controls (2.12,3) and (2.15,4) .. (2.18,4) -- (3.6,4);
\draw (3.6,4) .. controls (3.63,4) and (3.66,1) .. (3.69,1) -- (3.69,1) .. controls (3.72,1) and (3.75,4) .. (3.78,4) -- (4,4);
\end{tikzpicture}
\caption{Typical trajectories of the probability mass on the most favoured site, $\sup_{x \in \mathbb{Z}^+} P_\sigma(X_t = x)$, in the slowly varying case.}\label{trajpic}
\end{figure}

\begin{remark}
To give some intuition as to why the number of localisation sites depends on the second-order slow-variation rate $g$ in the way determined by \eqref{eq:N}, consider that, for $v > 1$,
\[ \mathbf{P} \left(\sigma_0 \in (u, uv] \:|\:  \sigma_0 > u \right) =  \frac{1/L(u) - 1/L(uv)}{1/L(u)} = 1 - \frac{L(u)}{L(uv)} \sim g(u) k(v)   ,  \]
and so $g(u)$ gives an approximate measure of how likely records, or near records, of the sequence $(\sigma_i)_{i \in \mathbb{N}}$ are to cluster on the same scale $u$. In particular, $g(u)^k$ gives the approximate probability that such a cluster consists of at least $k$ sites. Next, consider that the height of the $n^{\text{th}}$ record of the sequence $(\sigma_i)_{i \in \mathbb{N}}$ is approximately $L^{-1}(e^n)$ (see, e.g.\ Lemma \ref{lem:probrecord}). Hence, by a Borel-Cantelli argument, the summability of the sequence
\[  \left( d(e^n)^{k} \right)_{n \in \mathbb{N}}:=\left( g \left( L^{-1}(e^n) \right)^k \right)_{n \in \mathbb{N}}\]
determines whether a cluster of $k$ records, or near records, occurs around the $n^{\text{th}}$ record eventually almost-surely. From here, notice that a cluster of records, or near records, on the same scale naturally gives rise to a division of the probability mass function of the BTM across this cluster. Counting the site of the $n^{\text{th}}$ record and the site from which the BTM eventually escapes after leaving the associated cluster of records or near records, this line of argument suggests that, under Assumptions \ref{assumpt:sosv} and~\ref{assumpt:g22}, quenched localisation will occur on $N$ sites. With regard to the extra logarithmic factors appearing in the definition of $N$ in \eqref{eq:N} and in Assumption \ref{assumpt:g22}, it is possible that these are artifacts of our proof which could be removed (or at least relaxed).

The above heuristics allow us to conjecture the quenched localisation behaviour of the BTM on the integers. In particular, we expect that quenched localisation in the BTM on the integers occurs on a set of cardinality $N+1$, i.e.\ one larger than for the positive integers. The intuition is that the clustering argument described above is valid across the whole positive and negative half-lines. The extra localisation site takes into account the fact that the BTM can now escape, after leaving the cluster, in two directions. Nevertheless, formalising this heuristic presents additional technical challenges not present in the one-sided case, and we do not pursue this here.
\end{remark}

\begin{remark}
Let us draw a comparison with the BTM with regularly varying traps. In this case, it was recently shown in \cite{Croydon16} that
\[  \limsup_{t \to \infty} \sup_{x \in \mathbb{Z}} P_\sigma(X_t = x)  = 1    \quad \mathbf{P}\text{-almost-surely},\]
refining the original observation of \cite{FIN99} that the above $\limsup$ expression is strictly positive; note that here the probability measures $\mathbf{P}$ and $P_\sigma$ refer to the BTM on the integers. In other words, just as in the slowly varying case, there exist arbitrarily large times at which the probability mass of the BTM is, up to any specified error, completely localised. On the other hand, in \cite{Croydon16} it was also shown that in the regularly varying case
\[\liminf_{t \to \infty} \sup_{x \in \mathbb{Z}} P_\sigma(X_t = x)  = 0 \quad \mathbf{P}\text{-almost-surely},\]
in other words, there are also arbitrarily large times at which the BTM is completely delocalised.
\end{remark}

In light of Theorems \ref{thm:main1} and \ref{thm:main2}, it is natural to expect that the localisation set of the BTM is related to the set of `record traps'
\[  \mathcal{R} := \left\{  x \in \mathbb{Z}^+ : \sigma_x > \max_{0 \le y < x} \sigma_y \right\} . \]
The following result shows that the localisation set can actually be chosen to be a subset of $\mathcal{R}$ if and only if $N = 2$. It will become clear from our proofs that for $N > 2$ the localisation set will, at arbitrarily large times, also include certain `near records'.

\begin{theorem}[Criterion for localisation on record traps]
\label{thm:main3} Suppose Assumptions \ref{assumpt:sosv} and \ref{assumpt:g22} both hold. Then
\[ P_\sigma(X_t \in \mathcal{R} ) \to 1  \quad \mathbf{P}\text{-almost-surely} \]
if and only if $N = 2$.
\end{theorem}

\begin{remark}
The regimes in which two-site localisation occurs almost-surely, and indeed occurs on record traps, are precisely the regimes identified in \cite{Croydon15} and \cite{Kasahara86} in which simplified scaling limit theorems are available (see \cite[Remark 1.5]{Croydon15} and \cite[Remark 2.4]{Kasahara86}).
\end{remark}

Finally, we observe that each of our main results on localisation in the BTM with slowly varying traps is underpinned by the analogous result regarding the sum-max ratio for sequences of i.i.d.\ slowly varying random variables. As we noted above, with high probability, the partial sums of such sequences are asymptotically dominated by the maximum. However, it turns out that in general this is not the case almost-surely, as is demonstrated by the following theorem. This provides a crucial ingredient in our arguments for the BTM, and to the best of our knowledge has not appeared in the literature before.

\begin{theorem}[Sum/max ratio for sequences of slowly varying random variables]
\label{thm:main4} Let $(\sigma_i)_{i \in \mathbb{N}}$ be an i.i.d.\ sequence of copies of $\sigma_0$. Let $m_i$ and $S_i$ denote the maximum and sum respectively of the partial sequence $(\sigma_j)_{j \le i}$ . If Assumption \ref{assumpt:sosv} holds, then, almost-surely,
\[  \liminf_{i \to \infty} \frac{S_i}{m_i}   =   1  \quad \text{and} \quad \limsup_{i \to \infty} \frac{S_i}{m_i} \le N-1 . \]
Moreover, if Assumption \ref{assumpt:g22}(a) (with $N < \infty$) also holds, then
\[ \limsup_{i \to \infty} \frac{S_i}{m_i} = N-1 .   \]
\end{theorem}
\begin{remark}
In the special case of $L$ satisfying \eqref{eq:L}, this implies that
\[  \liminf_{i \to \infty} \frac{S_i}{m_i}   =   1  \quad \text{and} \quad \limsup_{i \to \infty} \frac{S_i}{m_i} = 1+ \left \lfloor \frac{\gamma}{1-\gamma} \right \rfloor  , \]
almost-surely. Hence, in this case, $ \lim_{i \to \infty} {S_i}/{m_i} = 1$ almost-surely if and only if $\gamma < 1/2$. For comparison, we recall that the latter limit holds in probability for all slowly varying tails \cite{Darling52}; an observation which (together with Fatou's lemma) already yields the $\liminf$ part of the previous result.
\end{remark}

\subsection{Outline of the paper}

The rest of this paper is organised as follows. In Section~\ref{sec:prelim} and Section~\ref{sec:svseq} we collect preliminary results that will be crucial in proving the main theorems. The results in Section~\ref{sec:prelim} relate to random walks in inhomogeneous trapping landscapes; the results in Section~\ref{sec:svseq} consider general properties of sequences of slowly varying random variables, including the key almost-sure bound on the ratio of the cumulative sum to the maximum stated as Theorem \ref{thm:main4} above.

In Section~\ref{sec:loc} we study upper bounds on the size of the quenched localisation set. We first give an explicit description of the localisation set $\Gamma_t$, and show that localisation does occur on this set eventually almost-surely. Here we also show that $|\Gamma_t| = 1$ with overwhelming probability, establishing the complete localisation in probability of Theorem~\ref{thm:maincl}. We next consider the almost-sure cardinality of $\Gamma_t$, showing that $|\Gamma_t| \le N$ eventually almost-surely, establishing the first claim of Theorem \ref{thm:main1} above. As a corollary, we show that, if $N = 2$, the localisation set $\Gamma_t$ is contained in the set of record traps $\mathcal{R}$, establishing one direction of Theorem \ref{thm:main3}.

Finally, in Section~\ref{sec:fav} we study the most favoured site (Theorem \ref{thm:main2} above). Since Theorem \ref{thm:maincl}, and the first part of Theorem \ref{thm:main1}, will already have been proved at this point, it will be sufficient to show that there exist arbitrarily large times at which the BTM is evenly balanced across a certain set of $N$ localisation sites. Since for $N \ge 3$, the above-mentioned $N$ sites are not all contained in the record traps $\mathcal{R}$, this will also establish the converse direction of Theorem \ref{thm:main3}.

\section{Random walks in inhomogeneous trapping landscapes}
\label{sec:prelim}

In this section we collect preliminary results on random walks in inhomogeneous trapping landscapes, in particular relating to hitting times and localisation. Note that in this section the trapping landscape $\sigma$ will always be assumed to be completely arbitrary and deterministic.

Throughout this section, fix $a, b \in \mathbb{Z}$ such that $a < b$ and let $X^{a, b}=(X^{a, b}_t)_{t \ge 0}$ denote the inhomogeneous continuous-time random walk (CTRW) in an arbitrary (deterministic) trapping landscape $\sigma=(\sigma_x)_{x \in [a, b]\cap \mathbb{Z}}$, with reflected boundary conditions at $a$ and $b$. More precisely, $X^{a, b}$ is the continuous-time Markov chain on $[a, b] \cap \mathbb{Z}$ with transition rates as at (\ref{trates}), where $y \sim x$ here means that $x$ and $y$ are nearest neighbours in $[a, b] \cap \mathbb{Z}$. Let $P^{a, b}_x$ denote the law of $X^{a, b}$ when started from the site $x \in [a, b] \cap \mathbb{Z}$. Henceforth, and throughout the rest of the paper, we shall refer to $X^{a, b}$ as the `inhomogeneous CTRW on $[a, b] \cap \mathbb{Z}$ in the trapping landscape $\sigma$'.

\subsection{Hitting times for inhomogeneous CTRWs}
We start by considering upper and lower bounds on the hitting times of the boundary by $X^{a, b}$. We let $\tau_a$ and $\tau_b$ denote the hitting time of the boundary at $a$ and $b$ respectively, i.e.\
\[ \tau_a := \inf\{ s: X_s^{a, b} = a \} \quad \text{and} \quad  \tau_b := \inf\{ s: X_s^{a, b} = b \} ,\]
and let $\tau$ be the hitting time of either boundary, i.e.\
\[ \tau := \min\{ \tau_a, \tau_b\}.\]

\begin{proposition}[Upper bounds on hitting times]
\label{prop:hittingtimeub}
For each $x \in [a, b] \cap \mathbb{Z}$ and $t > 0$,
\[    P^{a, b}_x \left( \tau_b \geq  t \right) \leq 2 t^{-1} \, (b-x) \sum_{a \le z < b} \sigma_z  .\]
Moreover, if ${S} \subseteq [x, b) \cap \mathbb{Z}$, then
\[    P^{a, b}_x ( \tau_b \geq t ) \leq 2 t^{-1} \bigg( (b-x) \sum_{ \{a \le z < b\} \setminus {S} } \sigma_z  +  |{S}| \max_{z \in {S}}\{b-z\}   \max_{z \in {S}} \{ \sigma_z \} \bigg)  .\]
\end{proposition}

\begin{proof}
By basic properties of simple random walks, the expected number of times the process $X^{a, b}$, started from $x$, visits a site $z \in (a, b) \cap \mathbb{Z}$ prior to $\tau_b$ is equal to $2 \min \{  b-z, b-x  \}$, with the mean holding time at each visit to $z$ being $\sigma_z$. Similarly, the expected number of times $X^{a, b}$ hits the site $a$ prior to $\tau_b$ is equal to $(b-x)$ with the mean holding time at each visit being $2 \sigma_a$. Hence
\[ E_x^{a, b} \left[ \tau_b \right]  =  \sum_{a \le z < b} 2  \min \{  b-z, b-x  \} \, \sigma_z .\]
For the first statement, we bound this expectation simply by
\[ E_x^{a, b} \left[ \tau_b \right]  \le  2 ( b-x) \, \sum_{a \le z < b}  \sigma_z ,\]
and apply Markov's inequality. For the second statement, we instead bound the expectation by
\[ E_x^{a, b} \left[ \tau_b \right]  \le  2 ( b-x) \, \sum_{\{a \le z < b\} \setminus {S} }  \sigma_z  +  2 |{S}| \max_{z \in {S}}\{b-z\}   \max_{z \in {S}} \{ \sigma_z \}  ,\]
and again apply Markov's inequality.
\end{proof}

\begin{proposition}[Lower bound on hitting times]
\label{prop:hittingtimelb}
For each $x \in [a, b) \cap \mathbb{Z}$, $z \in [x, b) \cap \mathbb{Z}$ and $t > 0$,
\[    P^{a, b}_x ( \tau_b \le  t) \leq \frac{t}{2 (b-z) \sigma_z } .\]
Moreover, for each $x \in (a, b) \cap \mathbb{Z}$ and $t > 0$,
\[    P^{a, b}_x (  \tau \le t) \leq \frac{t}{\min\{x-a,b-x\} \sigma_x } .\]
\end{proposition}

\begin{proof}
Consider the first statement, and note that $\tau_b$ is bounded below by the time spent by $X^{a, b}$ at the site $z$ prior to the time $\tau_b$. Assume for the moment that $x \neq a$. By basic properties of simple random walks, the number of times $X^{a, b}$, started from $x$, visits a site $z \in [x, b) \cap \mathbb{Z}$ prior to $\tau_b$ is distributed as a geometric random variable (supported on $\{1, 2, \ldots\}$) with mean $2(b-z)$. Moreover, the time spent at each visit is an independent exponential random variable with mean $\sigma_z$. Hence, the time spent at $z$ prior to $\tau_b$ is exponentially distributed with mean $2(b-z)\sigma_z$. This implies that
\[    P^{a, b}_x \left( \tau_b \le   t \right) \le  1 - \exp\left\{ - \frac{t}{2 (b-z) \sigma_z }    \right\} \leq\frac{t}{2 (b-z) \sigma_z } .\]
If $x = a$ the proof is identical, since the extra factors of two in the means of the number of visits and holding time distributions exactly cancel each other out.

The proof of the second statement is similar. This time, consider that the number of times $X^{a, b}$ visits $x$ prior to the time $\tau$ is distributed as a geometric random variable (supported on $\{1, 2, \ldots\}$) with mean
\[  2 \left( \frac{1}{x-a} + \frac{1}{b-x}  \right)^{-1}   \ge  \min \{ x-a, b-x \} , \]
and the result follows as before.
\end{proof}

\subsection{Localisation of inhomogeneous CTRWs}

To finish the section, we state a simple localisation property of $X^{a, b}$, expressed in terms of the trapping landscape $\sigma$.

\begin{proposition}
\label{prop:local}
For each $x \in [a, b] \cap \mathbb{Z}$ and ${S} \subseteq [a, b]\cap \mathbb{Z}$ such that $x \notin {S}$,
\[ \sup_{t \geq 0} P^{a, b}_x \left( X^{a, b}_t \in {S} \right) \leq  \frac{\sum_{z \in {S}} \sigma_z}{\sigma_x }  .\]
\end{proposition}
\begin{proof}
First note that, by standard continuous-time Markov chain theory, the process $X^{a, b}$ has a unique equilibrium distribution $\pi$. From the detailed balance equations, it is straightforward to check that $\pi \propto \sigma$, i.e. for each $z \in [a, b] \cap \mathbb{Z}$,
\[\pi(z) =\frac{\sigma_z}{\sum_{a \le y \le b}\sigma_y} . \]
Hence, using the reversibility of the process, for each $z \in [a, b] \cap \mathbb{Z}$ and time $t \ge 0$,
\[P^{a, b}_x  \left( X_t = z \right) = P^{a, b}_z  \left( X_t = x \right) \frac{\pi(z)}{\pi(x)} \le \frac{\pi(z)}{\pi(x)} = \frac{\sigma_z}{\sigma_x} . \]
Summing over ${S}$ yields the result.
\end{proof}

\section{Sequences of slowly varying random variables}
\label{sec:svseq}

In this section we collect results on i.i.d.\ sequences of copies of $\sigma_0$. Throughout this section we shall assume that $L$ is slowly varying, but we shall not necessarily assume that $L$ satisfies the second-order slow-variation in Assumption \ref{assumpt:sosv} or, indeed Assumption~\ref{assumpt:g22}, unless we specify this explicitly. We first prove preliminary results on general i.i.d.\ sequences of slowly varying random variable; these relate to exceedences, records, and partial sums of such sequences. We then apply these results to establish the almost-sure bound on the ratio of the sum to the maximum of the sequence $(\sigma_i)_{i \in \mathbb{N}}$ of Theorem \ref{thm:main4}. Finally, we study certain types of `hyperbolic' exceedence; the relevance of these exceedences to the BTM will be made clear in Section~\ref{sec:loc}.

Before proceeding, we shall first recall a useful consequence of second-order slow-variation Assumption \ref{assumpt:sosv} for certain expectations involving $\sigma_0$; the spirit is similar to that of de Haan's theorem (see \cite[Section 3.7]{Bingham87}). We also state a weaker version of the result which holds for general slowly varying tails; the spirit is similar to that of Karamata's theorem.

\begin{proposition}[See {\cite[Proposition 5.1]{Croydon15}}]
\label{prop:sosv}
Assume $L$ satisfies Assumption \ref{assumpt:sosv}. Let $f:(0,\infty) \to \mathbb{R}^+$ be a continuously differentiable function and $I \subseteq (0,\infty)$ an interval (which may be unbounded), and suppose there exists a $\delta > 0$ for which: (i) $f(t)\mathbf{1}_{\{t\in I\}} = o(t^\delta)$ as $t \rightarrow 0$; and (ii) both $f'(t) t^\delta$ and $f'(t) t^{-\delta}$ are integrable over the interval $I$. Then the function
\[ \Gamma(n) := \mathbf{E} \left[ f(\sigma_0 / n) \mathbf{1}_{ \{ \sigma_0/n \in I \}} \right]\]
satisfies
\[\lim_{n \to \infty} \frac{L(n) \Gamma(n)}{g(n)} =- \lambda  \int_I f'(t) \log t \, dt\]
for some constant $\lambda > 0$ that only depends on $L$ and $g$. Moreover, even if $L$ does not satisfy Assumption \ref{assumpt:sosv} but is still slowly varying, we have that
\[\lim_{n \to \infty} L(n) \Gamma(n) = 0 .\]
\end{proposition}
\begin{remark}
Note that the first statement of this result is a very slight generalisation of \cite[Proposition 5.1]{Croydon15}, which is recovered by setting $I := (0,\infty)$. It is proved in an identical manner. The second statement is also proved in an identical manner.
 \end{remark}

\subsection{Preliminary results: Exceedences, records, and partial sums}
\label{subsec:prelim}

Our preliminary results on general sequences of slowly varying random variables are split into three categories, containing results on: (i) first exceedences of levels; (ii) records of the sequence; and (iii) partial sums. Note that some of the results in this section hold only if Assumption \ref{assumpt:sosv} is satisfied; Assumption \ref{assumpt:g22} will not be relevant to this section.

Note that when we describe a collection $(X_i)_{i\in I}$ of non-negative random variables as being bounded above or bounded below in probability we mean that $(X_i)_{i\in I}$ or $(1/X_i)_{i\in I}$, respectively, is tight.

\subsubsection{First exceedences of levels}
For a level $x > 0$, let $i_x$ denote the index of the first exceedence of $x$ in the sequence $(\sigma_i)_{i \in \mathbb{N}}$, that is $ i_x := \min \{ i : \sigma_i > x \}$. Further, denote by
$ i_x^- := {\rm{argmax}} \{ \sigma_i : i < i_x  \}$.

\begin{lemma}[Typical exceedences]
\label{lem:iprob}
As $x \to \infty$,
 \[   \frac{i_x}{ L(x) } ,\quad \frac{L(\sigma_{i_x})}{L(x)} - 1,\quad \frac{L(\sigma_{i^-_x})}{L(x)}   \quad \text{and} \quad \ 1-\frac{L(\sigma_{i^-_x})}{L(x)}   \]
are all bounded above and below in probability.
\end{lemma}

\begin{proof} We first note that
\[\sum_{i\in\mathbb{N}}\delta_{(\frac{i}{n},\frac{L(\sigma_i)}{n})}\rightarrow \nu=\sum_{i}\delta_{(u_i,v_i)},\]
in distribution as random measures on $\mathbb{R}^+\times\mathbb{R}^+$, where $\nu$ is a Poisson random measure with intensity $v^{-2}dudv$. (Here $\delta_{(u,v)}$ is the probability measure placing all its mass at $(u,v)$.) It follows that
\begin{equation}\label{poisson}
\min\left\{\frac{i}{n}:\:\frac{L(\sigma_i)}{n}>1\right\}\rightarrow \inf\left\{u_i:\:v_i>1\right\}
\end{equation}
in distribution, where the limit is a $(0,\infty)$-valued random variable. Taking $n=L(x)$ in the above yields $i_x/L(x)$ converges in distribution, and so is bounded above and below in probability. Moreover, ${L(\sigma_{i_x})}/{L(x)}$ converges in distribution to the $v_i$ such that $(u_i,v_i)$ is an atom of $\nu$ and $u_i$ obtains the infimum in (\ref{poisson}). Since the latter is a $(1,\infty)$-valued random variable, the second claim holds. Similarly, ${L(\sigma_{i^-_x})}/{L(x)}$ converges to maximum $v_j$ such that $(u_j,v_j)$ is an atom of $\nu$ and $u_j$ is strictly less than the infimum on the right-hand side of (\ref{poisson}). Since this is a $(0,1)$-valued random variable, the proof is complete.
\end{proof}

\begin{lemma}[Level/exceedence ratio]
\label{lem:ex}
Assume $L$ satisfies Assumption \ref{assumpt:sosv}. Then there exists a $c > 0$ such that, as $x \to \infty$,
\[ \mathbf{E} \left[  \sigma_{i_x}^{-1} \right] < c \,  x^{-1} g(x)   \]
eventually.
\end{lemma}

\begin{proof}
We compute as follows
\[\mathbf{E} \left[  \sigma_{i_x}^{-1} \right] = L(x) \, \mathbf{E}   \left[  \sigma_0^{-1}  \mathbf{1}_{\{\sigma_0 > x \} } \right]
  =  x^{-1} L(x) \, \mathbf{E}   \left[ f(\sigma_0/x)\mathbf{1}_{\{\sigma_0/x > 1 \}}  \right] ,\]
 where $f(x) := x^{-1} $. Applying the first statement of Proposition \ref{prop:sosv} we deduce that
 \[    \mathbf{E}   \left[ f(\sigma_0/x) \mathbf{1}_{\{\sigma_0/x > 1 \}} \right] \sim \frac{ \lambda g(x) }{L(x)} \int_1^\infty t^{-2} \log t \, dt , \]
which yields the result.
\end{proof}

\subsubsection{Records of the sequence}

For $n \in \mathbb{N}$, let $r_n$ indicate the index of the $n^{\rm{th}}$ record of the sequence $(\sigma_i)_{i \in \mathbb{N}}$, and abbreviate $\sigma_{(n)} := \sigma_{r_n}$.

\begin{lemma}[Bounds for records]
\label{lem:probrecord}
Assume $L$ satisfies Assumption \ref{assumpt:sosv}. Then for each $\varepsilon > 0$ there exists a $c > 0$ such that
\[ \mathbf{P} \left( \log L(\sigma_{(n)}) \notin n(1-\varepsilon, 1+\varepsilon)  \right) < c n^{-2} \]
and
\[  \mathbf{P} \left(   r_n < L(\sigma_{(n-1)}) / n^2  \quad \text{or} \quad r_n >  2 L(\sigma_{(n-1)}) \log n  \right) < c n^{-2} \]
hold for all $n$. In particular, as $n \to \infty$,
\[ \log L(\sigma_{(n)})  \sim n \quad \text{and}  \quad  L(\sigma_{(n-1)}) / n^2 \leq r_n\leq 2 L(\sigma_{(n-1)}) \log n  \]
eventually almost-surely.
\end{lemma}
\begin{proof}
For the first statement, note that the continuity of $L$ (guaranteed by Assumption~\ref{assumpt:sosv}) ensures that the sequence $(\log L(\sigma_i))_{i \in \mathbb{N}}$ consists of unit mean exponentially distributed random variables. By the memoryless property of the exponential distribution, the gaps in the sequence $(\log L(\sigma_{(i)}))_{i\in\mathbb{N}}$ are therefore also unit mean exponentially distributed random variables. Hence the statement is just a standard large deviation bound for exponentially distributed random variables. For the second statement, note that
\begin{eqnarray*}
\mathbf{P} ( r_n  < L(\sigma_{(n-1)}) / n^2\:|\:\sigma_{(n-1)})  &\le& \mathbf{P} ( r_n - r_{n-1} < L(\sigma_{(n-1)}) / n^2\:|\:\sigma_{(n-1)})\\
& \le&  \frac{L(\sigma_{(n-1)}) }{n^2 }\frac{1}{L(\sigma_{(n-1)})} = n^{-2}
 \end{eqnarray*}
 by the union bound, and
 \begin{equation}\label{tb}
 \mathbf{P} ( r_n  > 2 L(\sigma_{(n-1)}) \log n\:|\:\sigma_{(n-1)}) =  \left( 1 - \frac{1}{L(\sigma_{(n-1)}) } \right)^{2 L(\sigma_{(n-1)}) \log n }\leq n^{-2} .
 \end{equation}
Hence, taking expectations, we have that
\[ \mathbf{P} ( r_n  < L(\sigma_{(n-1)}) / n^2 \quad \text{or} \quad r_n > 2 L(\sigma_{(n-1)}) \log n) \leq 2n^{-2},\]
which yields the result. Finally, the last statement is just an application of the Borel-Cantelli lemma.
\end{proof}

\begin{lemma}[Location/exceedence ratio for records]
\label{lem:record}
For each $k \in \mathbb{N}$ and $n \in \mathbb{N}$,
\[  \mathbf{E} \left[  \left(   \frac{r_{n}}{L(\sigma_{(n-1)})}   \right)^{k}   \right]  < k!  .\]
\end{lemma}

\begin{proof} Similarly to (\ref{tb}), for each $\lambda > 0$ we have
\[\mathbf{P}\left(r_n>\lambda L(\sigma_{(n-1)})|\: \sigma_{(n-1)}\right) = \left(1-\frac{1}{L(\sigma_{(n-1)})}\right)^{\lambda L(\sigma_{(n-1)})} <  e^{-\lambda}.\]
Hence $\mathbf{P}(r_n>\lambda L(\sigma_{(n-1)})) < e^{-\lambda}$, in other words, the random variable $r_n/L(\sigma_{(n-1)})$ is stochastically dominated by a mean one exponential random variable. The $k$th moment of the latter is precisely $k!$, and so we are done.
\end{proof}

\begin{lemma}[Ratio of successive records] \label{momentcomparison}
Assume $L$ satisfies Assumption \ref{assumpt:sosv}. Then there exists a $c > 0$ and $x_0<\infty$ such that, almost-surely: if $x\geq x_0$ and $n\geq m$, then
\[ \mathbf{E}\left(\frac{\sigma_{(m)}}{\sigma_{(n)}}\:\vline\:\sigma_{(m)}\right)\mathbf{1}_{\{\sigma_{(m)}\geq x\}}
< c^{n-m} \,  g(x)^{n-m}.\]
\end{lemma}

\begin{proof} Choose $x_0$ large enough so that the bound of Lemma \ref{lem:ex} holds for $x\geq x_0$, and let $c$ be the constant of that lemma. Then, applying Lemma \ref{lem:ex} repeatedly, we find that, almost-surely, if $x\geq x_0$ and $m\leq n$, then
\begin{eqnarray*}
\mathbf{E}\left(\frac{\sigma_{(m)}}{\sigma_{(n)}}\:\vline\:\sigma_{(m)}\right)\mathbf{1}_{\{\sigma_{(m)}\geq x\}}
&=&\mathbf{E}\left(\frac{\sigma_{(m)}}{\sigma_{(n-1)}}\mathbf{E}\left(\frac{\sigma_{(n-1)}}{\sigma_{(n)}}\:\vline\:\sigma_{(n-1)}\right)
\:\vline\:\sigma_{(m)}\right)\mathbf{1}_{\{\sigma_{(m)}\geq x\}}\\
& <  &\mathbf{E}\left(\frac{\sigma_{(m)}}{\sigma_{(n-1)}} \, c  \, g( \sigma_{(n-1)}) \:\vline\:\sigma_{(m)}\right)\mathbf{1}_{\{\sigma_{(m)}\geq x\}}\\
& < & c  \, g( x) \, \mathbf{E}\left(\frac{\sigma_{(m)}}{\sigma_{(n-1)}} \:\vline\:\sigma_{(m)}\right)\mathbf{1}_{\{\sigma_{(m)}\geq x\}}\\
&&\vdots\\
&\leq & c^{n-m} g(x)^{n-m},
\end{eqnarray*}
as desired. Note that we use the monotonicity of $g$ (guaranteed by Assumption \ref{assumpt:sosv}) to deduce that $g( \sigma_{(n-k)})\leq g(x)$ for $k=1,\dots, n-m$.
\end{proof}

\begin{lemma}[Sum of records]
 \label{sumsofrecords}
Assume $L$ satisfies Assumption \ref{assumpt:sosv}, and recall the definition $d(u) := g(L^{-1}(u))$. Then for each $\varepsilon>0$ sufficiently small, integer $k\geq1$ and positive sequence $\delta_n$, there exists a $c > 0$ such that
\[\mathbf{P}\left(\sum_{i=1}^{n-1}\sigma_{(i)} \geq (k-1+ \delta_n)\sigma_{(n)}\right) < c n^{-2} + c \, \delta_n^{-1} \, d(e^{n(1-\varepsilon)})^k\]
for all $n$.
\end{lemma}

\begin{proof}
Define the event $\mathcal{A}_n :=\{ \sigma_{(n-k)} \geq L^{-1}(e^{n(1-\varepsilon)})\}$. Then, by Lemma \ref{lem:probrecord} and Markov's inequality, there exists a $c > 0$ such that
\begin{eqnarray*}
\mathbf{P} \left(\sum_{i=1}^{n-1}\sigma_{(i)} \geq (k-1+ \delta_n)\sigma_{(n)}\right) & <  & c n^{-2}+\mathbf{P}\left(\sum_{i=1}^{n-k}\sigma_{(i)} \geq \delta_n \sigma_{(n)}, \, \mathcal{A}_n \right)\\
& <  &cn^{-2}+ \delta_n^{-1}\sum_{i=1}^{n-k}\mathbf{E}\left(\frac{\sigma_{(i)}}{\sigma_{(n)}} \mathbf{1}_{\{\mathcal{A}_n \}} \right).
\end{eqnarray*}
The lower bound for $\sigma_{(n-k)}$ that holds on $\mathcal{A}_n$ allows us to apply Lemma \ref{momentcomparison} to deduce that this is eventually bounded above, for some $c_1 > 0$, by
\begin{align*}
&  \qquad \qquad < c_1 n^{-2} + c_1 \, \delta_n^{-1}  \, d(e^{n(1-\varepsilon)})^k \, \sum_{i=1}^{n-k}
\mathbf{E}\left(\frac{\sigma_{(i)}}{\sigma_{(n-k)}} \mathbf{1}_{\{\mathcal{A}_n \}}  \right).
\end{align*}
Next, we note that the summands are bounded as follows
\[\mathbf{E}\left(\frac{\sigma_{(i)}}{\sigma_{(n-k)}} \mathbf{1}_{\{\mathcal{A}_n \}}  \right)\leq \mathbf{E}\left(\frac{\sigma_{(i)}}{\sigma_{(n-k)}} \right) \leq \mathbf{E}\left(\frac{\sigma_{(i)}}{\sigma_{(n-k)}} \mathbf{1}_{\{\sigma_{(i)}\geq x\}}\right)+\mathbf{P}\left(\sigma_{(i)}\leq x\right).\]
For the first term, we again apply Lemma~\ref{momentcomparison} to deduce that, for large enough $x$,
\[\mathbf{E}\left(\frac{\sigma_{(i)}}{\sigma_{(n-k)}} \mathbf{1}_{\{\sigma_{(i)}\geq x\}}\right)\leq (c g(x))^{n-k-i}.\]
In particular, by the fact that $g(x)\rightarrow 0$, we can choose $x$ such that $cg(x)<1$. Moreover, it is clear that, for each fixed $x$, there exists a constant $c_x$ such that
\[\mathbf{P}\left(\sigma_{(i)}\leq x\right)\leq\mathbf{P}\left(\log L(\sigma_{(i)})\leq \log L(x)\right) = \mathbf{P}\left(\mathrm{Po}( \log L(x))\geq i\right)\leq c_xe^{-i},\]
where we use that $(\log L(\sigma_{(i)}))_{i\geq 1}$ simply represents the points of a unit rate Poisson process, and denote by $\mathrm{Po}(\lambda)$ a Poisson random variable with parameter $\lambda$. Hence, we conclude that
\[\mathbf{P}\left(\sum_{i=1}^{n-1}\sigma_{(i)} <  (k-1+ \delta_n)\sigma_{(n)}\right) < c n^{-2} + c \, \delta_n^{-1} \,  d(e^{n(1-\varepsilon)})^k \sum_{i=1}^{n-k}\left((c g(x))^{n-k-i} +c_xe^{-i}\right),\]
and since the sum is bounded in $n$, this completes the proof.
\end{proof}

\subsubsection{Partial sums}

For $i$ an index, and let $(\sigma_{i}^{(1)}, \sigma_{i}^{(2)}, \ldots, \sigma_{i}^{(i)})$ be the (descending) order statistics of the subsequence $\{\sigma_j\}_{1 \le j \le i}$. Moreover, let $S_{i}^{(k)}$ denote the sum of the collection~$(\sigma_{i}^{(j)})_{k \le j \le i}$.

\begin{lemma}[Bound on sum below a level]
\label{lem:sumlevel}
Fix a $k \in \mathbb{N}$, and let $\ell_i, \delta_i$ be positive sequences such that $\ell_i \to \infty$ as $i \to \infty$. Then, there exists a $c > 0$ such that, as $i \to \infty$,
\[ \textbf{P} \left( S_{i}^{(k)} > \ell_i \delta_i\:\vline\:\sigma_1,\dots,\sigma_i\leq l_i  \right) <  c \, \delta_i^{-1} i^k \mathbf{E}\left(f_k\left(\frac{\sigma_0}{\ell_i}\right)\mathbf{1}_{\{\sigma_0/\ell_i\leq 1\}}\right)^k \]
eventually, where $f_k(x):=x^{1/k}$.
\end{lemma}
\begin{proof} Let $i\geq k$. By symmetry and Markov's inequality we have that
\begin{eqnarray*}
\lefteqn{\mathbf{P}\left(S_{i}^{(k)}> \ell_i \delta_i,\:\sigma_1,\dots,\sigma_i \leq \ell_i \right)}\\
&\leq & i^{k-1}\mathbf{P}\left(S_{i}^{(k)} > \ell_i \delta_i,\:\sigma_{i}^{(j)}=\sigma_j\mbox{ for }j=1,\dots,k-1,\:\sigma_{i}^{(1)}\leq \ell_i \right)\\
& = & i^{k-1}\mathbf{P}\left(\sum_{j = k}^i \sigma_j > \ell_i \delta_i,\: \sigma_{k-1} \leq \sigma_{k-2}\leq\dots\leq\sigma_1\leq \ell_i, \:  \sigma_j \le \sigma_{k-1} \mbox{ for }j=k,\dots,i  \right)\\
&\leq & i^{k-1}  \delta_i^{-1}  \mathbf{E}\left(\sum_{j=k}^i \frac{\sigma_j}{\ell_i}  \mathbf{1}_{\{\sigma_{k-1}\leq\sigma_{k-2}\leq\dots\leq\sigma_1\leq \ell_i\}}  \mathbf{1}_{\{\sigma_j\leq \sigma_{k-1} \text{ for }j=k,\dots,i   \}}\right)  \\
& \leq & i^{k}  \delta_i^{-1}   \mathbf{E}\left( \frac{\sigma_k}{\ell_i}  \mathbf{1}_{\{\sigma_k \le \sigma_{k-1}\leq\sigma_{k-2}\leq\dots\leq\sigma_1\leq \ell_i \}} \mathbf{1}_{\{  \sigma_j\leq\ell_i \text{ for }j=k+1,\dots,i   \}}\right) .
\end{eqnarray*}
Using the independence of the $(\sigma_j)_{j\geq 1}$, the above expectation is equal to
\begin{align*}
&  \mathbf{E}\left( \frac{ \sigma_k }{\ell_i} \mathbf{1}_{\{\sigma_k \le \sigma_{k-1}\leq\sigma_{k-2}\leq\dots\leq\sigma_1\leq \ell_i\}}  \right) \mathbf{P}\left(\sigma_0\leq \ell_i \right)^{i-k}  \\
 & \qquad \qquad \leq  \mathbf{E}\left(\prod_{j=1}^k\left(\frac{\sigma_j}{\ell_i}\right)^{1/k}\mathbf{1}_{\{\sigma_j\leq \ell_i\}}\right) \mathbf{P}\left(\sigma_0\leq \ell_i \right)^{i-k} \\
& \qquad \qquad = \mathbf{E}\left(f_k\left(\frac{\sigma_0}{\ell_i}\right)\mathbf{1}_{\{\sigma_0/\ell_i\leq 1\}}\right)^k \mathbf{P}\left(\sigma_0\leq \ell_i \right)^{i-k} .
\end{align*}
On dividing through by $\mathbf{P}(\sigma_1,\dots,\sigma_i\leq \ell_i )=(1-L(\ell_i)^{-1})^{i}$, we thus obtain \[\mathbf{P}\left(S_{i}^{(k)}>\ell_i \delta_i \:\vline\:\sigma_1,\dots,\sigma_i\leq \ell_i \right)\leq {i^{k}  \delta_i^{-1}} \mathbf{E}\left(f_k\left(\frac{\sigma_0}{\ell_i}\right)\mathbf{1}_{\{\sigma_0/\ell_i\leq 1\}}\right)^k\left(1-L(\ell_i)^{-1}\right)^{-k},\]
which proves the result.
\end{proof}

\begin{corollary}
\label{cor:sumlevel}
Assume $L$ satisfies Assumption \ref{assumpt:sosv}. Fix a $k \in \mathbb{N}$, and let $\ell_i, \delta_i$ be positive sequences such that $\ell_i \to \infty$ as $i \to \infty$. Then there exists a $c > 0$ such that, as $i \to \infty$,
\[ \textbf{P} \left( S_{i}^{(k)} > \ell_i \delta_i\:\vline\:\sigma_1,\dots,\sigma_i \le \ell_i  \right) < c  \, \delta_i^{-1} \left( \frac{i}{L(\ell_i)} g(\ell_i) \right)^k \]
eventually.
\end{corollary}
\begin{proof}
By the first statement of Proposition \ref{prop:sosv} we deduce that, for some $\lambda > 0$,
\[\mathbf{E}\left(f_k\left(\frac{\sigma_0}{\ell_i}\right)\mathbf{1}_{\{\sigma_0/\ell_i\leq 1\}}\right) \sim -\frac{\lambda \, g(\ell_i)}{L({\ell_i})} \int_{0}^1 f_k'(t)\log t \, dt,\]
from which the result follows by applying Lemma \ref{lem:sumlevel}.
\end{proof}

\subsection{The ratio of the sum to the maximum}

In this section we prove the key almost-sure bound on the ratio of the sum to the maximum of the sequence $(\sigma_i)_{i \in \mathbb{N}}$ in Theorem~\ref{thm:main4}. Throughout this section, we assume that $L$ satisfies Assumption \ref{assumpt:sosv}, and define $N$ as in~\eqref{eq:N}.

Observe that Theorem \ref{thm:main4} is a consequence of the following two results, which we prove in the next two subsections. Recall that $m_i$ and $S_i$ denote the maximum and partial sum of the partial sequence $(\sigma_j)_{j \le i}$, and $S_{i}^{(k)}$ is the sum from the $k$th largest term.

\begin{proposition}[Upper bound on sum/max ratio]
\label{prop:ub} Suppose Assumption \ref{assumpt:sosv} holds. For each $\varepsilon > 0$, as $i \to \infty$,
\[ \frac{S_i}{m_i} < N - 1 + \varepsilon \]
eventually almost-surely. Moreover, if Assumption \ref{assumpt:g22}(b) (with $N < \infty$) also holds, then additionally, as $i \to \infty$,
\[ \frac{S_i^{(N)}}{m_i} < \frac{\varepsilon}{\log i} \]
eventually almost-surely.
\end{proposition}

\begin{proposition}[Lower bound on sum/max ratio]
\label{prop:lb}
If Assumptions \ref{assumpt:sosv} and \ref{assumpt:g22}(a) (with $N < \infty$) hold, then for each $\varepsilon > 0$ sufficiently small, as $i \to \infty$,
\[ \frac{S_i}{m_i} > N - 1 - \varepsilon \]
holds infinitely often.
\end{proposition}

\begin{remark}
Of course, the control on the rate of convergence in the second part of Proposition \ref{prop:ub} under Assumption \ref{assumpt:g22}(b) is stronger than we need for Theorem \ref{thm:main4}, but this extra control will be useful in Section \ref{sec:fav}.
\end{remark}

\subsubsection{Upper bound on sum/max ratio}

Recall, for each $n \in \mathbb{N}$, the notation $r_n$ and  $\sigma_{(n)}$ for the location and magnitude of the $n$th record from Section \ref{subsec:prelim}, and let $S_{(n)}^-$ denote the sum of the collection $\{\sigma_i\}_{i < r_n}$. Further, for each $\varepsilon > 0$ and $n\in\mathbb{N}$ define the events
\[  \mathcal{A}^\varepsilon_n := \left\{  \frac{ S_{(n)}^- }{  \sigma_{(n-1)} } > N-1 +  \varepsilon \right\},\qquad  \bar{\mathcal{A}}^\varepsilon_n := \left\{  \frac{ S^{(N)}_{r_n-1}}{  \sigma_{(n-1)} } > \frac{ \varepsilon} {\log n} \right\} .  \]
Since the ratio of the sum to the max is increasing up until new records of the sequence, to establish Proposition \ref{prop:ub}, by the Borel-Cantelli lemma it is sufficient to prove the following.

\begin{lemma}\label{anlem} Suppose Assumption \ref{assumpt:sosv} holds. For each $\varepsilon > 0$ we have that
\[ \sum_{n \in \mathbb{N}} \mathbf{P} \left(  \mathcal{A}_n^\varepsilon  \right) < \infty .\]
Moreover, if Assumption \ref{assumpt:g22}(b) (with $N < \infty$) also holds, then additionally
\[ \sum_{n \in \mathbb{N}} \mathbf{P} \left(  \bar{\mathcal{A}}_n^\varepsilon  \right) < \infty .\]
\end{lemma}

\begin{proof} By definition, we have that
\[S_{(n)}^- =\sum_{i=1}^{n-1}\sigma_{(i)}+\sum_{i=1}^{r_n-1}\sigma_i\mathbf{1}_{\{i\not\in\mathcal{R}\}},\]
where we recall that $\mathcal{R}$ is the collection of record traps $(r_n)_{n\geq 1}$. Now, conditional on $\{(r_i,\sigma_{(i)}):\:i\leq n\}$, the traps that contribute to the second sum are independent. Moreover, for $i\in (r_{m-1},r_{m})$, $m\in\{1,\dots,n\}$, we have that the traps are distributed as $\sigma_0|\{\sigma_0\leq \sigma_{(m-1)}\}$, and so are stochastically dominated by $\sigma_0|\{\sigma_0\leq \sigma_{(n-1)}\}$. It follows that
\[\mathbf{P}\left(\sum_{i=1}^{r_n-1}\sigma_i\mathbf{1}_{\{i\not\in\mathcal{R}\}}\geq \lambda\sigma_{(n-1)}\:\vline\:(r_i,\sigma_{(i)}):\:i\leq n\right) \leq F \left(r_n, \sigma_{(n-1)}, \floor{\lambda}+1, \lambda - \floor{\lambda} \right),\]
where, recalling the notation for $S^{(k)}_r$ from Section \ref{subsec:prelim},
\begin{equation}\label{Fdef}
 F(r, l, k, \delta):=\mathbf{P}({S}^{(k)}_{r}\geq  l \delta | \sigma_1, \ldots, \sigma_r \le l) .
 \end{equation}
Applying the above reasoning, we have
\begin{eqnarray*}
\lefteqn{\mathbf{P} \left(  \mathcal{A}_n^\varepsilon  \right)}\\
& \leq &
\sum_{k=0}^{N-1}\mathbf{P} \left(\sum_{i=1}^{n-1}\sigma_{(i)}\geq \left(k + \varepsilon/2  \right) \sigma_{(n-1)} , \:\sum_{i=1}^{r_n-1} \sigma_i\mathbf{1}_{\{i\not\in\mathcal{R}\}}\geq
\left(N-2-k + \varepsilon/2 \right)\sigma_{(n-1)}\right)\\
&\leq &\sum_{k=0}^{N-1}\mathbf{E}\left(\mathbf{1}_{\{\sum_{i=1}^{n-2}\sigma_{(i)}\geq (k-1+\varepsilon/2 )\sigma_{(n-1)}\}} F \left((r_n,\sigma_{(n-1)},N-1-k, \varepsilon/2  \right) \right).
\end{eqnarray*}
Recall that, by Lemma \ref{lem:probrecord} and the union bound, there exists a $c > 0$ such that, for sufficiently large $n$,
\[  \mathbf{P} \left(L(\sigma_{(n-1)}) < e^{n(1- \varepsilon)} \, , \ r_n > 2 L(\sigma_{(n-1)})\log n \right) <  c n^{-2} . \]
Applying Lemma \ref{sumsofrecords} and Corollary \ref{cor:sumlevel} (with $\delta_i = \varepsilon/2$), it follows that there exists a $c > 0$ such that, for sufficiently large $n$,
\begin{eqnarray*}
\mathbf{P} \left(  \mathcal{A}_n^\varepsilon  \right)
&<  & c n^{-2} + c \,  \sum_{k=0}^{N-1}\left( d( e^{n(1-\varepsilon)}) \log n \right)^{N-1-k}
\mathbf{P}\left(\sum_{i=1}^{n-2}\sigma_{(i)}\geq (k-1+\varepsilon/2)\sigma_{(n-1)}\right)\\
& <  &c n^{-2}+ c \left( d(e^{n(1-\varepsilon)}) \log n \right)^{N-1} .
\end{eqnarray*}
Considering the definition of $N$ in \eqref{eq:N}, and noting that a monotonic sequence $a_n$ is summable if and only if $a_{\floor{(1-\varepsilon) n}}$ is summable, this completes the proof of the first statement.

For the second statement, a similar argument holds. In particular, we start by noting that
\[S_{r_n-1}^{(N)}=\min_{k=1,\dots,N-1}\left\{\sum_{i=1}^{(n-1)-(N-k)}\sigma_{(i)}+\Sigma^{(k)}\right\},\]
where $\Sigma^{(k)}$ is the sum $\sum_{i=1}^{r_n-1}\sigma_i\mathbf{1}_{\{i\not\in\mathcal{R}\}}$ with the largest $k-1$ terms excluded. Since
\[\left\{S_{r_n-1}^{(N)}>\varepsilon\sigma_{(n-1)}/\log n,\:\Sigma^{(k)}<\varepsilon\sigma_{(n-1)}/2\log n\right\}\subseteq\left\{\sum_{i=1}^{(n-1)-(N-k)}\sigma_{(i)}>\varepsilon\sigma_{(n-1)}/2\log n\right\},\]
decomposing the probability space into regions where $\Sigma^{(k)}>\varepsilon\sigma_{(n-1)}/2\log n\geq\Sigma^{(k+1)}$ thus yields
\begin{eqnarray*}
{\mathbf{P} \left( \bar{\mathcal{A}}_n^\varepsilon  \right)}&\leq &\mathbf{P} \left(\sum_{i=1}^{(n-1)-(N-1)}\sigma_{(i)}>\varepsilon\sigma_{(n-1)}/2\log n \right)\\
&&+\sum_{k=1}^{N-1}\mathbf{E}\left(\mathbf{1}_{\{\sum_{i=1}^{(n-1)-(N-k-1)}\sigma_{(i)}\geq \varepsilon\sigma_{(n-1)}/2\log n\}} F \left(r_n,\sigma_{(n-1)},k, \varepsilon/2\log n  \right) \right).
\end{eqnarray*}
Noting that the proof of Lemma \ref{sumsofrecords} shows that
\[\mathbf{P} \left(\sum_{i=1}^{(n-1)-j}\sigma_{(i)}>\varepsilon\sigma_{(n-1)}/2\log n \right) < cn^{-2}+cd(e^{n(1-\varepsilon)})^j\log n,\]
uniformly for $j=1,\dots,N-1$, and applying Corollary \ref{cor:sumlevel} with $\delta_i = \varepsilon/2\log i$ similarly to above, it follows that
\[\mathbf{P} \left( \bar{\mathcal{A}}_n^\varepsilon  \right) < cn^{-2}+
cd(e^{n(1-\varepsilon)})^{N-1}\left(\log n\right)^{N},\]
which is summable on Assumption~\ref{assumpt:g22}(b).
\end{proof}

\subsubsection{Lower bound on sum/max ratio}

Again recall, for each $n \in \mathbb{N}$, the notation $r_n$ and $\sigma_{(n)}$ from Section~\ref{subsec:prelim}, and let $\{  r_n^{(1)}, r_n^{(2)}, \ldots , r_n^{(N-2)}  \}$ denote the indices of the largest $N-2$ terms of the sequence lying between $r_n$ and $r_{n+1}$ in increasing order of index (if there are insufficient terms, set the undefined terms to be equal to $r_{n+1}$), abbreviating $\sigma_{(n)}^{(i)} := \sigma_{r_n^{(i)}}$. To establish the lower bound we consider the following event, defined for each $\varepsilon \in (0, 1)$ and $n\in\mathbb{N}$,
\[ \mathcal{A}^{\varepsilon}_n :=   \bigcap_{1 \le i \le N-2}   \left\{ \sigma^{(i)}_{(n)} / \sigma_{(n)} \in  (1 - \varepsilon, 1) \right\} .\]
Clearly on the event $\mathcal{A}^\varepsilon_n$ we have that
\[ \frac{ S_{r_n^{(N-2)}} }{m_{r_n^{(N-2)}}}    >  (N-1)(1-\varepsilon), \]
and so in order to establish Proposition \ref{prop:lb} it is sufficient to prove that for each $\varepsilon \in (0, 1)$ the events $\mathcal{A}_n^\varepsilon$ hold infinitely often. That this is true can be deduced by applying the following lemma in conjunction with the conditional Borel-Cantelli lemma.

\begin{lemma}
Let $\mathcal{F}_n$ denote the $\sigma$-algebra generated by $\sigma_1,\sigma_2,\dots,\sigma_{r_{n+1}}$. Then $\mathcal{A}_n^\varepsilon\in \mathcal{F}_n$, and, on Assumptions \ref{assumpt:sosv} and \ref{assumpt:g22}(a) (with $N < \infty$),
\[  \sum_n \mathbf{P}(\mathcal{A}^\varepsilon_{n} | \mathcal{F}_{n-1}) = \infty\]
almost-surely.
\end{lemma}

\begin{proof} That $\mathcal{A}_n^\varepsilon\in \mathcal{F}_n$ is clear by definition. For the second claim, we note that the event $\mathcal{A}^\varepsilon_n$ is just the event that the first $N-2$ exceedences of the level $\sigma_{(n)} (1 - \varepsilon)$ after $r_n$ do not also exceed the level $\sigma_{(n)}$. Hence
\begin{eqnarray*}
\mathbf{P}(\mathcal{A}^\varepsilon_{n}\: |\: \mathcal{F}_{n-1}) &=&\mathbf{P}\left(\sigma_0<u\:|\:\sigma_0>(1-\varepsilon)u\right)^{N-2}|_{u=\sigma_{(n)}}\\
&=&\left(1-\frac{L((1-\varepsilon)\sigma_{(n)})}{L(\sigma_{(n)})}\right)^{N-2}.
\end{eqnarray*}
Combining with the bounds on $\sigma_{(n)}$ from Lemma \ref{lem:probrecord} and the definition of second-order slow-variation, it follows that there exists a $c > 0$ such that
\[\mathbf{P}(\mathcal{A}^\varepsilon_{n}\: |\: \mathcal{F}_{n-1})
\sim \left(-k(1-\varepsilon)g(\sigma_{(n)})\right)^{N-2} \ge c d(e^{n(1+\varepsilon)} )^{N-2}\]
eventually, where we have used the eventual monotonicity of $g$ guaranteed by Assumption~\ref{assumpt:sosv}. Again noting that a monotone sequence $a_n$ is summable if and only if $a_{\floor{(1+\varepsilon) n}}$ is summable, the above sequence is not summable on Assumption \ref{assumpt:g22}(a), completing the proof.
\end{proof}

\subsection{Hyperbolic exceedences}

In this section we study certain `hyperbolic' exceedences; the relevance of these exceedences to the localisation of the BTM will be made clear in Section \ref{sec:loc}. Note that we do not apply Assumptions \ref{assumpt:sosv} or \ref{assumpt:g22} here.

Before defining these hyperbolic exceedences, recall that $i_x$ indicates the first exceedence of a level $x > 0$ of the sequence $(\sigma_i)_{i \in \mathbb{N}}$, and $S_i$ denotes the cumulative sum of the sequence up to an index~$i$. Define an auxiliary function $h_t$ such that $h_t \to \infty$ as $t \to \infty$ sufficiently slowly such that
\begin{align}
\label{eq:h}
  \frac{L(t h^3_t)}{L(t)} < 1 + \frac{1}{h_t} \quad \text{and} \quad  \frac{L(t /h_t^3)}{L(t)} > 1 - \frac{1}{h_t}
\end{align}
eventually. We note that the choice of such a $h_t$ is possible for any slowly varying function $L$; see \cite{Croydon15} for an explicit construction. We define the \textit{hyperbolic exceedences} of the level~$t$ to be the sites
\begin{equation}\label{jtdef}
j_t:= \min \{ i : i S_i>t/h_t\} \quad \text{and} \quad  j_t^- := {\rm{argmax}} \{ \sigma_i : i \leq j_t  \} .
\end{equation}

\begin{lemma}[Typical hyperbolic exceedences]
\label{lem:jprob}
For each $t \ge 0$, denote
\begin{equation}\label{elltdef}
 \ell_t :=  \min \{s \ge 0: s L(s) \ge t \} ,
 \end{equation}
which is well-defined since $L$ is c\`{a}dl\`{a}g. Then, as $t \to \infty$,
\[ \mathbf{P} \left( i_{\ell_t}=j_t=j_t^- \right) \to 1 .\]
\end{lemma}

\begin{proof} First note that, by our choice of $h_t$ in \eqref{eq:h} and applying Lemma \ref{lem:iprob}, we have that, as $t \to \infty$, $i_{\ell_t}  > L(\ell_t)   h^{-1}_{\ell_t}$ and $\sigma_{i_{\ell_t}} >  \ell_t h^3_{\ell_t}$ both hold with high probability. Then it is clear that $j_t \le i_{\ell_t}$, since then, with high probability,
\[  i_{\ell_t} S_{i_{\ell_t}} > i_{\ell_t} \sigma_{i_{\ell_t}} > L(\ell_t)  h^{-1}_{\ell_t} \times \ell_t h^3_{\ell_t} > t.\]
For the other direction, consider the proceeding record site $r_{n_t}$ prior to $i_{\ell_t}$. Applying Lemma \ref{lem:iprob} again, with high probability we have that
\[  i_{\ell_t} < L(\ell_t) h_{\ell_t}^{1/2},  \qquad \sigma_{r_{n_t}}=\sigma_{i_{\ell_t}^-} < \ell_t / h^3_{\ell_t} \qquad \text{and}\qquad L(\sigma_{i_{\ell_t}^-} )>L(\ell_t)h_{\ell_t}^{-1/3}. \]
Thus, using the notation $F$ from (\ref{Fdef}) and writing $\Gamma(x):=\mathbf{E}(\frac{\sigma_0}{x}\mathbf{1}_{{\sigma_0}\leq{x}})$,
\begin{eqnarray*}
\mathbf{P}\left(S_{i_{\ell_t} - 1} / \sigma_{r_{n_t}}>h_{\ell_t}\right)&=& \mathbf{E}\left(F(i_{\ell_t}-2,\sigma_{i_{\ell_t}^-},1,h_{\ell_t}-1)\right)\\
&\leq& \mathbf{E}\left(\min\left\{\frac{cL(\ell_t)\Gamma(\sigma_{i_{\ell_t}^-})}{\sqrt{h_{\ell_t}}},1\right\}\right)+o(1), \end{eqnarray*}
where to deduce the inequality we have applied Lemma \ref{lem:sumlevel}. Now, by Proposition \ref{prop:sosv}, $\Gamma(x)=o(L(x)^{-1})$. Hence it follows that
\[ \mathbf{P}\left(S_{i_{\ell_t} - 1} / \sigma_{r_{n_t}}>h_{\ell_t}\right)\leq
\mathbf{E}\left(\min\left\{\frac{cL(\ell_t)}{\sqrt{h_{\ell_t}}L(\sigma_{i_{\ell_t}^-})},1\right\}\right)+o(1)\leq c h_{\ell_t}^{-1/6}+o(1)\rightarrow 0.\]
This implies that with high probability eventually
\[   (i_{\ell_t}-1) S_{i_{\ell_t}-1} <  h_{\ell_t}i_{\ell_t} \sigma_{r_{n_t}} <  L(\ell_t) h_{\ell_t}^{3/2} \times \ell_t / h_{\ell_t}^3 =  t/h_{\ell_t}^{3/2} < t/h_{t}.\]
Thus we have shown that $\mathbf{P}(i_{\ell_t}=j_t)\to 1$ as $t\to \infty$. To complete the proof, we simply note that $i_{\ell_t}=j_t$ implies $j_t$ is a record, and therefore $j_t=j_t^-$.
\end{proof}

\section{Quenched localisation on $N$ sites}
\label{sec:loc}

In this section we establish that localisation takes place on no more than $N$-sites almost-surely, that is, we prove the first claim of Theorem \ref{thm:main1}. Note that this claim only has content if $N$ is finite, so wherever we work under Assumption \ref{assumpt:sosv} in this section we shall always assume that $N < \infty$. On the other hand, Assumption \ref{assumpt:g22} will not play a role in this section.

 We begin in Section \ref{subsec:def} by giving an explicit construction of the localisation set $\Gamma_t$. We note that there are several different possible approaches to defining this set; one way, for instance, would be to select a certain set of $N$ sites, whereby the cardinality of $\Gamma_t$ would be bounded by construction. We choose a construction that makes no explicit reference to cardinality; instead $\Gamma_t$ is defined to include all points lying in a certain region of $(x, \sigma_x)$-space. The advantage of this definition is that the subsequent proof that the BTM localises on $\Gamma_t$ is straightforward; we do this in Section \ref{subsec:loc}. The trade-off is that bounding the cardinality of $\Gamma_t$ is no longer straightforward, and requires a somewhat lengthy computation; we undertake this computation in Section \ref{subsec:card}. Along the way, we also establish that the localisation set $\Gamma_t$ consists of a single site with overwhelming probability, and moreover that $\Gamma_t$ is contained on the record traps eventually almost-surely if $N = 2$, hence completing the proof of Theorem \ref{thm:maincl} and the forward direction of Theorem \ref{thm:main3}.

\subsection{Defining the localisation set}
\label{subsec:def}

We shall define the localisation set $\Gamma_t$ by first considering a certain region $\mathcal{G}_t \subseteq \mathbb{Z}^+ \times \mathbb{R}^+$ defined by an outer boundary $O_t$ and a lower boundary $D_t$; the localisation set will then consist of all points $(x, \sigma_x)$ that lie inside this region. Before we define these explicitly, we first motivate our construction, which is based around certain record traps $z^{I}_t$ and $z^{O}_t$ such that $0 < z^{I}_t \le z^{O}_t$. First, we construct the site $z^{I}_t$ to be the furthest record site from the origin such that the BTM is overwhelmingly likely to have visited this site by time $t$, essentially because this site is not too far from the origin and the traps lying between it and the origin are sufficiently shallow. Naturally this construction demands that we use an upper bound on hitting times of the BTM, and indeed we approximate the time until the BTM exits a certain region by the product of the sum of the traps in the region and the length of the region. This suggests that the site  $z^{I}_t$ should be defined via the notion of {hyberbolic exceedences} (see Section \ref{sec:svseq}).

Second, we construct the site $z^{O}_t$ to be the furthest record site from the origin such that it is {possible} for the BTM to reach this site by time $t$. This time we use a lower bound for the hitting times of the BTM, approximating the time until the BTM exits a certain region by the product of the length of the region and the depth of the deepest trap in the region. We optimise our construction of $z^{O}_t$ by a process of `chaining' from the initial site $z^{I}_t$, using the observation that each new record trap that is visited by the BTM by time $t$ {reduces} the successive distance that the BTM is able to travel by time~$t$, since the BTM now has to overcome the holding time associated with this new record trap. Note that it is possible (and indeed will turn out to be likely) that the site $z^{O}_t$ is actually the same as the site $z^{I}_t$; this occurs if the chaining terminates at the first stage.

Once we have defined the sites $z^{I}_t$ and $z^{O}_t$, we construct the outer boundary $O_t$ to appropriately contain these sites, i.e. such that $0  < z^{I}_t \le z^{O}_t < O_t$. This is done in such a way to guarantee that the BTM is located within this boundary at time $t$; again we use a lower bound for hitting times of the BTM. Our construction also guarantees that $z^{O}_t$ is the largest trap located in the region $[0,O_t)$. Finally, we make use of the observation that the BTM, at any given time $t$, is more likely to be located in a deep trap than a shallow trap; the lower boundary $D_t$ represents the trap depth above which we know that the BTM is located with high probability.

We now formalise the above heuristics to give an explicit definition of the localisation set $\Gamma_t$.  In the following definitions we make reference to a certain auxiliary scaling function $h_u \to \infty$ as $u \to \infty$. We think of $h_u$ as being arbitrarily slowly growing, and indeed we shall insist that $h_u$ grows sufficiently slowly to satisfy certain conditions. First, similarly to \eqref{eq:h}, we shall require that, for any $k \in \mathbb{N}$,
\begin{align}
\label{eq:h2}
 \frac{L(t h^k_t)}{L(t)} < 1 + \frac{1}{h_t} \quad \text{and} \quad  \frac{L(t /h_t^k)}{L(t)} > 1 - \frac{1}{h_t}
\end{align}
for large $t$. Further, we shall also require that
\begin{align}
\label{eq:h3}
 {h^4_t}{L(t h_t^2)} \mathbb{E}\left[ \frac{\sigma_0}{t h_t^2}  \mathbf{1}_{\{ \sigma_0 < t h_t^2 \}} \right]  \to 0,
\end{align}
remarking that this is possible by the second statement of Proposition \ref{prop:sosv}. When working under Assumption \ref{assumpt:sosv} (with $N < \infty$), defining $\hat{h}_n := h_{2 N e^{2n} L(e^{2n}) \log n}$,
 we shall additionally require that
\begin{align}
\label{eq:h4}
 \sum_{n \in \mathbb{N}} e^{-n/2} \hat{h}_n < \infty \quad \text{and} \quad  \sum_{n \in \mathbb{N}}  \left( d(e^n) (\log n) (\hat{h}_n)^5 \right)^{N-1} < \infty.
\end{align}
Finally, for technical reasons, we shall also require that $h_u$ is continuous. The relevance of these conditions will become clear later in the section; for now, note simply that it is always possible to choose such an $h$ (see \cite{Croydon15} for remarks as to an explicit construction).

We first define the sites $z^{I}_t$ and $z^{O}_t$. The site $z^{I}_t$ is taken to be the record site $j_t^-$, as introduced at \eqref{jtdef}. To define the site $z^{O}_t$, define iteratively
\[ y_t^1:= z^{I}_t \ , \quad y_t^{i+1} := \min \{ z\in (y_t^{i},y_t^{i} +  h_t \max\{t / \sigma_{y_t^i}, 1\}) :\: \sigma_z > \sigma_{y_t^{i}} \} \]
until this chain terminates. The site $z^{O}_t$ is defined to be the last site so-defined by the chaining (which is possibly, and indeed probably, the same as the site $z^{I}_t$). Note that this method of `chaining' can be thought of as a general procedure that starts from a certain site $z^{I}_t$ and occurs at a certain time $t$; we will refer back to this general method of `chaining' in Section \ref{subsec:card}. We also note that in the process of chaining we make use of the lower bound $\max\{ \cdot, 1\}$ when extending the outer boundary from $y_t^{i}$ to $y_t^{i+1}$; this is done for technical reasons, essentially to ensure that the regions we are considering are always growing (albeit arbitrarily slowly) with $t$. This will allow us to successfully apply our holding time estimates to these regions.

We now define the localisation set $\Gamma_t$, by specifying an outer boundary $O_t$ and a lower boundary $D_t$ for the localisation region $\mathcal{G}_t$. First, define the outer boundary $O_t$ to be
\[ O_t :=  z^{O}_t + h_t \max\{ t  / \sigma_{z^{O}_t}, 1 \}  \, .\]
To define the lower boundary $D_t$, for a level $\ell > 0$ denote the quantity $S^\ell_i := \sum_{ z < i: \sigma_z < \ell } \sigma_z $, and set
\[  D_t := \max\{ \ell \ge 0 : S^\ell_{O_t} < \sigma_{z^{I}_t} / h_t  \} . \]
Note that this construction of the lower boundary can be thought of as a general procedure that is given by a certain boundary $O_t$ and level $\sigma_{z^{I}_t} / h_t$; we will refer back to this general procedure in Section \ref{subsec:card}.
We can now define the localisation set to be the point set
\[ \Gamma_t := \{ x \in \mathbb{Z} : (x, \sigma_x) \in \mathcal{G}_t \}   \]
where $\mathcal{G}_t := \{ (x, \sigma_x) :   x < O_t , \sigma_x \ge D_t \}$. Figure \ref{gtpic} shows typical and atypical configurations of this set.

\begin{figure}[ht]
\begin{tikzpicture}
\draw[thick,->] (0,0) -- (2.1, 0) node[anchor=north] {\small $O_t$} -- (4.5,0) node[anchor=north west] {\small $\mathbb{Z}$};
\draw[thick,->] (0,0) -- (0,1.6) node[anchor=east] {\small $\ell_t$}  -- (0,2) node[anchor=east] {\small $D_t$}  -- (0,4.5) node[anchor=south] {\small $L(\sigma_z)$};
\draw[thick] (0,1.6) -- (-0.1,1.6);
\draw[thick] (0,2) -- (-0.1,2);
\draw[thick] (1.6,0) -- (1.6,-0.1);
\draw[thick] (2.1,0) -- (2.1,-0.1);
\draw (0.1,4) .. controls (0.15,1.5) and (0.2,1.5) .. (4.3,1.5) ;
\draw[dashed] (2.1, 2) -- (2.1, 4);
\draw[dashed] (0, 2) -- (2.1, 2);
\foreach \Point in {(0.5,0.08), (1.1,0.76), (1.6,1.96), (2.6,2.76), (3.3,3.46)}{
    \node at \Point {\textbullet};
}
\draw (3, 1.7) node {\textbullet};
\draw[thin] (0,0) -- (0.5,0) -- (0.5, 0.1) -- (1.1, 0.1) -- (1.1, 0.8) -- (1.6, 0.8) -- (1.6, 2) -- (2.6, 2) -- (2.6, 2.8) -- (3.3, 2.8) -- (3.3, 3.5) -- (4.0, 3.5) ;
\draw (1, 3) node[anchor=south] {\small $\mathcal{G}_t$};
\draw (1.05, 2.2) node[anchor=south] {\small $\Gamma_t$};
\draw[thick,->] (1.27, 2.38) -- (1.53, 2.13) ;
\draw (4.5, 1.55) node[anchor=north] {\small $\text{hyp}_t$};
\draw (0.9, 0.1) node[anchor=north] {\small $z_t^I = z_t^O$} ;
\end{tikzpicture}
\hspace{0.3cm}
\begin{tikzpicture}
\draw[thick,->] (0,0)   -- (4.5,0) node[anchor=north west] {\small $\mathbb{Z}$};
\draw[thick,->] (0,0) -- (0,2) node[anchor=east] {\small $\ell_t$}    -- (0,4.5) node[anchor=south] {\small $L(\sigma_z)$};
\draw[thick] (0,1.7) -- (-0.1,1.7);
\draw[thick] (0,2) -- (-0.1,2);
\draw[thick] (1.6,0) -- (1.6,-0.1);
\draw[thick] (3.3,0) -- (3.3,-0.1);
\draw[thick] (3.4,0) -- (3.4,-0.1);
\draw (0.1,4) .. controls (0.15,2) and (0.2,2) .. (4.3,2) ;
\draw[dashed] (3.4, 1.7) -- (3.4, 4);
\draw[dashed] (0, 1.7) -- (3.4, 1.7);
\foreach \Point in {(0.5,0.08), (1.1,0.76), (1.6,1.96), (2.6,2.76), (3.3,3.46)}{
    \node at \Point {\textbullet};
}
\draw (3, 1.7) node {\textbullet};
\draw[thin] (0,0) -- (0.5,0) -- (0.5, 0.1) -- (1.1, 0.1) -- (1.1, 0.8) -- (1.6, 0.8) -- (1.6, 2) -- (2.6, 2) -- (2.6, 2.8) -- (3.3, 2.8) -- (3.3, 3.5) -- (4.0, 3.5) ;
\draw (1.6, 3) node[anchor=south] {\small $\mathcal{G}_t$};
\draw (3.8, 2.3) node[anchor=south] {\small $\Gamma_t$};
\draw[thick,->] (3.7, 2.85) -- (3.5, 3.25) ;
\draw[thick,->] (3.5, 2.6) -- (3, 2.65) ;
\draw[thick,->] (3.5, 2.5) -- (2.2, 2.13) ;
\draw[thick,->] (3.7, 2.4) -- (3.3, 2.1) ;
\draw (4.5, 2.05) node[anchor=north] {\small $\text{hyp}_t$};
\draw (1.6, 0.1) node[anchor=north] {\small $z_t^I$} ;
\draw (3, 0.1) node[anchor=north] {\small $z_t^O$} ;
\draw (3.6, 0) node[anchor=north] {\small $O_t$} ;
\draw (0,1.7) node[anchor=east] {\small $D_t$};
\end{tikzpicture}
\caption{The localisation set at: (i) a typical time; and (ii) a slightly later atypical time (in particular a `relocalisation time'; see Section \ref{subsec:card}). Depicted as well are (after rescaling by $L$) successive records and near records of $\{\sigma_z\}$, the cumulative sum process $S_i$, the hyperbola $\text{hyp}_t := \{(z, \sigma_z) : z \sigma_z = t\}$, and the level $\ell_t$ as defined as in Lemma \ref{lem:jprob}.}\label{gtpic}
\end{figure}
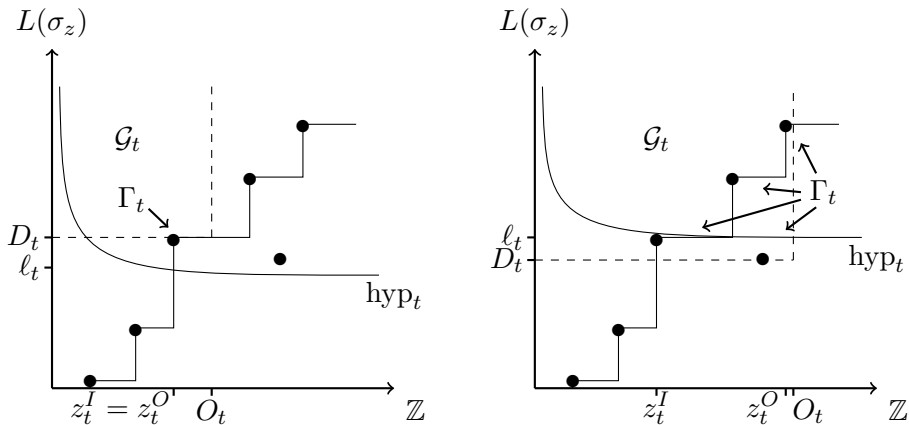

To complete this section, we prove that the localisation site consists of a single site with overwhelming probability; the precise description of this single site closely mirrors the construction in \cite{Muirhead15}. Note that this result does not require Assumption \ref{assumpt:sosv}, and holds for any slowly varying $L$.

\begin{proposition}
For each $t \ge 0$ define the level $ \ell_t$ as at (\ref{elltdef}), and denote the site $Z_t :=i_{\ell_t}$ (in the notation of Section \ref{subsec:prelim}). Then, as $t \to \infty$,
\[ \mathbf{P} \left( \Gamma_t = Z_t \right) \to 1 . \]
\end{proposition}

\begin{proof} Since $Z_t=i_{\ell_t}$ and $z_t^{I}=j^-_t$ we have by Lemma \ref{lem:jprob} that $\mathbf{P} ( Z_t = z_t^{I})\rightarrow 1$. From this it immediately follows that $\mathbf{P} ( Z_t \subseteq \Gamma_t  ) \to 1$. Hence, it is sufficient to show that, as $t \to \infty$,
\[ \mathbf{P} ( | \Gamma_t | = 1) \to 1. \]
To this end, we need to show that with high probability neither of the two disjoint regions
\[ R_1 :=  \{ x:\:  x < O_t,\: \sigma_x > \sigma_{z^{I}_t} \}  \quad \text{and} \quad  R_2 :=   \{x:  x < O_t,\: D_t \le \sigma_x \le \sigma_{z^{I}_t} \} \setminus \{(z_t^I, \sigma_{z_t^I})\} . \]
contains a point.

Observe that
\[\mathbf{P}\left(\exists x\in (Z_t, Z_t + h_t \max\{t / \sigma_{Z_t},1\})\mbox{ such that }\sigma_x>\sigma_{Z_t}|\:\sigma_{Z_t}\right)\leq  \frac{h_t \max\{t / \sigma_{Z_t},1\}}{L(\sigma_{Z_t})}.\]
If $\sigma_{Z_t} > \ell_t h_{\ell_t}^2$, which by Lemma \ref{lem:iprob} and our assumptions on $h_t$ in \eqref{eq:h2} holds with high probability, then the right-hand side is bounded above by
\[\max\left\{\frac{h_t}{L(\ell_t h_{\ell_t}^2)},\frac{1}{h_t}\right\},\]
and, possibly choosing an even more slowly growing $h$ than already required by our assumptions in \eqref{eq:h2}, one can check that this converges to 0 as $t\rightarrow\infty$. Since we also know from the previous paragraph that $Z_t = z_t^{I}$ with high probability, it follows that with high probability the chaining procedure terminates at the first step and we have $\mathbf{P}(Z_t = z_t^{I}=z_t^{O})\rightarrow 1$ as $t\rightarrow\infty$. In particular, this implies $\mathbf{P}(R_1=\emptyset)\rightarrow1$.

We now deal with $R_2$. For this, note that $Z_t = z_t^{I}=z_t^{O}$ implies
\[ (0 ,O_t)\subseteq (0, Z_t + h_t \max\{t / \sigma_{Z_t},1\}) . \]
Moreover, note that if $\sigma_{Z_t} > \ell_t h_{\ell_t}^2$ and $L(\ell_t) / h_{\ell_t}< Z_t$, then
\begin{align*}
h_t \max\{t / \sigma_{Z_t},1\} & \leq h_t\max\left\{t/\ell_th_{\ell_t}^2,1\right\} \le
 h_t\max\left\{L(\ell_t)/h_{\ell_t}^2,1\right\} \\& \leq  h_t\max\left\{Z_t/h_{\ell_t},1\right\}\leq 2Z_t,
 \end{align*}
where for the purposes of the final inequality, we suppose that $h$ is so slowly varying as to satisfy $h_t\sim h_{\ell_t}$.

Since the assumptions hold with high probability (by applying our previous observations and Lemma \ref{lem:iprob} again), it follows that
\begin{equation}\label{decay1}
\mathbf{P}\left(Z_t = z_t^{I},\:(0 ,O_t)\subseteq (0,3 Z_t)\right)\to 1
\end{equation}
as $t\to\infty$. Now, define
\[\tilde{S}_t := \sum_{\{x < 3Z_t\} \setminus \{z_t^I\} } \sigma_x\mathbf{1}_{\{\sigma_x \le \sigma_{Z_t}\}}.\]

Similarly to the proof of Lemma \ref{anlem}, applying Lemma \ref{lem:sumlevel} we have that eventually
\[\mathbf{P}\left(\tilde{S}_t > h_{\ell_t}^{-1}\sigma_{Z_t}|\:(Z_t,\sigma_{Z_t})\right) < h_{\ell_t}^2 Z_t\Gamma(\sigma_{Z_t}),\]
where $\Gamma(x):=\mathbf{E}(\frac{\sigma_0}{x} \mathbf{1}_{ \{ \sigma_0 < x\} })$. (For this it is useful to note that, conditional on $(Z_t,\sigma_{Z_t})$, if $x< Z_t$, then $\sigma_x$ is distributed as $\sigma_0|\{\sigma_0 \le \ell_t\}\prec \sigma_0|\{\sigma_0 \le\sigma_{Z_t}\}$, whereas if $x> Z_t$, then $\sigma_x$ is simply a copy of $\sigma_0$ and so $\sigma_x\mathbf{1}_{\{\sigma_x \le \sigma_{Z_t}\}}\prec \sigma_0|\{\sigma_0 \le \sigma_{Z_t}\}$, and all the relevant traps are independent.) Considering our assumptions on $h_t$ in \eqref{eq:h3}, along with the fact that $\sigma_{Z_t} > \ell_t h_{\ell_t}^2$ and $Z_t < L(\ell_t) h_t$ with high probability by Lemma \ref{lem:iprob}, this implies that with high probability
\[\mathbf{P}\left(\tilde{S}_t > h_t^{-1}\sigma_{Z_t}|\:(Z_t,\sigma_{Z_t})\right)  < \frac{L(\ell_t) }{h_{\ell_t}L(\ell_t h_{\ell_t}^2)}.\]
Since the upper bound here converges to $0$ as $t \to \infty$, we thus deduce $\mathbf{P}(\tilde{S}_t > h_{\ell_t}^{-1}\sigma_{Z_t})\to 0$. Combining this with (\ref{decay1}), we find that
\[\mathbf{P}\left(R_2\neq \emptyset\right) \rightarrow 1,\]
which completes the proof.
\end{proof}

In conjunction with Proposition \ref{prop:loc} (established in the next section), the previous result yields Theorem \ref{thm:maincl}.

\subsection{Localisation on the localisation set}
\label{subsec:loc}

In this section we prove that localisation occurs on the set $\Gamma_t$ eventually almost-surely. The argument follows a similar structure to that used to show localisation of one-dimensional random walk in random environments in \cite[Theorem 2.5.3]{Zeitouni}, for example, with the distinction that in our case the localisation set can consist of more than one point. In particular, we prove the following.

\begin{proposition}
\label{prop:loc}
As $t \to \infty$,
\[ P_\sigma(X_t \in \Gamma_t) \to 1 \quad \mathbf{P}\text{-almost-surely.}\]
\end{proposition}

Before we prove Proposition \ref{prop:loc}, we first define some notation. For each $t > 0$ and trapping landscape $\sigma$, define the random times
\[  \tau_t^1  :=   \min \{s : X_s =  z^{I}_t \} \quad \text{and} \quad  \tau_t^2  :=  \min \{ s > \tau_t^1 : X_s \ge O_t  \} . \]
Proposition \ref{prop:loc} is then an easy consequence of the following three lemmas.

\begin{lemma}[Hitting the localisation set]
As $t \to \infty$,
\[ P_\sigma(\tau_t^1 \le t ) \to 1   \quad \mathbf{P}\text{-almost-surely}  . \]
\end{lemma}
\begin{proof}
Applying the upper bound on hitting times in the first statement of Proposition~\ref{prop:hittingtimeub} (with $a = x = 0$ and $b = z_t^I$),
\[P_\sigma \left( \tau^1_t  \le t  \right) >  1 - 2 t^{-1} z_t^I \sum_{z < z_t^I} \sigma_z .\]
By the definition of the site $z^{I}_t$, the sum of traps here is less than ${t}/{(z^{I}_t-1)h_t}$. It is a simple exercise to check that $z_t^{I}\geq 2$ for large $t$ almost-surely, and thus the result follows.
\end{proof}

\begin{lemma}[Staying within the boundary]
\label{lem:stay}
As $t \to \infty$,
\[ P_\sigma(\tau_t^2 - \tau^1_t > t) \to 1  \quad \mathbf{P}\text{-almost-surely} . \]
\end{lemma}
\begin{proof} First, define the hitting time $\tau^{O}_t :=  \min\{s  : X_s= z^{O}_t \}$, which satisfies $\tau^O_t \ge \tau_t^1$. Applying the lower bound on hitting times in the first statement of Proposition \ref{prop:hittingtimelb} (with $a = 0$, $b = O_t$ and $x = z_t^O$) yields that
\[  P_\sigma(\tau_t^2 - \tau^1_t \le t) \le P_\sigma \left(\tau^2_t-\tau^{O}_t   \le t   \right) < \frac{t}{2 (O_t - z_t^O) \sigma_{z_t^O}}  ,\]
By the definition of the localisation set $\Gamma_t$
\[  O_t - z^{O}_t = h_t \max \{ t/ \sigma_{z^{O}_t}, 1 \} \ge h_t  \, t / \sigma_{z^{O}_t},  \]
and the result follows.
\end{proof}

\begin{lemma}[Localisation on the localisation set]
As $t \to \infty$,
\[ P_\sigma\left( X_t \in \Gamma_t | \:\tau_t^1 < t < \tau_t^2-\tau_t^1 \right) \to 1  \quad \mathbf{P}\text{-almost-surely}. \]
\end{lemma}

\begin{proof} For a given $t$, let $(\hat X^t_s)_{s\geq 0}$ be the inhomogeneous CTRW on $[0 ,O_t]$ in the trapping landscape $\sigma$ (see Section \ref{sec:prelim} for the definition of this Markov process), started from $z^{I}_t$, and let $\hat{P}$ denote its law. Then we have that $(X_{(s+\tau_t^1)\wedge \tau_t^2})_{s\geq 0}$ has the same distribution as $(\hat X^t_{s\wedge \tau_t^2})_{s\geq 0}$. In particular, applying the Markov property at~$\tau_t^1$, we have
\begin{eqnarray*}
{P_\sigma\left(X_t\not\in\Gamma_t|\:\tau_t^1 < t < \tau_t^2-\tau_t^1 \right)}
&\leq&\sup_{s\leq t} \hat{P} \left(\hat{X}^t_s\not\in \Gamma_t|\: \tau_t^2>t\right)\\
&\leq&\frac{\sup_{s\leq t} \hat{P} \left(\hat{X}^t_s\not\in\Gamma_t\right)}{P_\sigma\left({\tau}_t^2-\tau_t^1>t \right)}.
\end{eqnarray*}
From Lemma \ref{lem:stay}, we know the denominator converges to one almost-surely. Moreover, applying the localisation result of Proposition \ref{prop:local} (and the definition of the lower boundary $D_t$), we have that the numerator converges to zero, so we are done.
\end{proof}

\subsection{The cardinality of the localisation set}
\label{subsec:card}

In this section we establish that $|\Gamma_t| \le N$ eventually almost-surely. The fact that $X_t$ is contained on the record traps $\mathcal{R}$ eventually almost-surely if $N = 2$ will follow as an easy corollary. Throughout this section we shall work on Assumption \ref{assumpt:sosv} and additionally assume that $N < \infty$.

Similarly to the proof of Theorem \ref{thm:main4} in Section \ref{sec:svseq} above, we prove that $|\Gamma_t| \le N$ by defining a certain sequence of events for the countable sequence $\mathcal{R}$ of record sites. Broadly speaking, this event is whether there are more than $N$ sites in $\Gamma_t$ at the times when $|\Gamma_t|$ is at a local maximum. Considering the construction of $\Gamma_t$, it can be seen that $|\Gamma_t|$ is at a local maximum precisely at the `relocalisation times' between successive record traps. So let us define this event.

Recall the sequence $\mathcal{R} := (r_n)_{n\in\mathbb{N}}$ of record traps, and that $S_{(n)}^-$ denotes the sum of the traps prior to $r_n$. Set $t_0=0$, and for $n\geq 1$, define the time $t_n$ to satisfy the equation
\[ t_n / h_{t_n} = S_{(n)}^- (r_n-1)  , \]
which is well-defined since we insisted that $h_t$ be continuous. For simplicity, we abbreviate $h_{(n)} := h_{t_n}$. Note that for $t\in[t_{n-1},t_n)$, we have that $z^{I}_t=r_{n-1}$. Hence the time $t_n$ represents the `relocalisation time' between the record traps at $r_{n-1}$ and $r_n$.

Observe further that for $t\in[t_{n-1},t_n)$, the boundary $O_t$ is strictly increasing. In particular, it will be sufficient to check that $|\Gamma_t|  \le N $ at the instants immediately prior to $t \in \mathcal{T}:=\cup_{n\geq 1}\{ t_n \}$ (at least for large $n$, almost-surely). So for each $n$ define $O_{(n)}$ to be the outer boundary obtained by starting at site $r_{n-1}$ and chaining according to the procedure introduced in Section \ref{subsec:def} at time $t_{n}$. Moreover, define the lower boundary $D_{(n)}$ using the general construction from the boundary $O_{(n)}$ and level $\sigma_{(n-1)}$. Then it is clear that for all such times $t \in [t_{n-1}, t_n)$ we have $\Gamma_t\subseteq \{r_{n-1},r_n\}\cup(\cup_{i=1}^2R_i)$, where $(R_i)_{i=1}^2$ are given by the disjoint sets (depicted in Figure \ref{gtpic2})
\begin{eqnarray*}
R_1&:=&\left\{x:\:x\in (r_n,O_{(n)}),\:\sigma_x > \sigma_{(n-1)}\right\},\\
R_2&:=&\left\{x:\:x < O_{(n)},\: D_{(n)} \le \sigma_x < \sigma_{(n-1)}\right\} .
\end{eqnarray*}
Note that here we are assuming $L$ is continuous, by Assumption \ref{assumpt:sosv}. Hence, to show that $|\Gamma_t|\leq N$ for large $t$ almost-surely, it will suffice to show that $|R_1| + |R_2| \leq N-2$ for large $n$ almost-surely. We begin by studying the cardinality of~$R_1$.

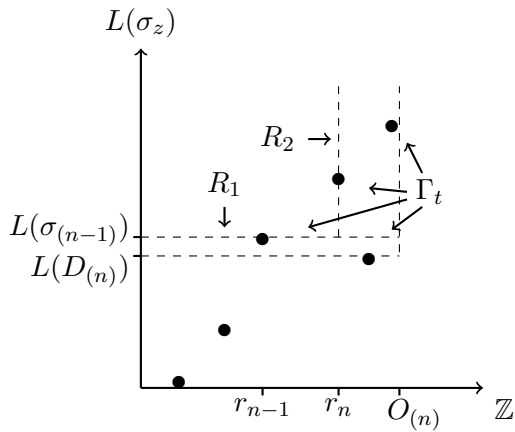
\begin{figure}[t]
\begin{tikzpicture}
\draw[thick,->] (0,0)   -- (4.5,0) node[anchor=north west] {\small $\mathbb{Z}$};
\draw[thick,->] (0,0)   -- (0,4.5) node[anchor=south] {\small $L(\sigma_z)$};
\draw[thick] (0,1.75) -- (-0.1,1.75);
\draw[thick] (0,2) -- (-0.1,2);
\draw[thick] (1.6,0) -- (1.6,-0.1);
\draw[thick] (2.6,0) -- (2.6,-0.1);
\draw[thick] (3.4,0) -- (3.4,-0.1);
\draw[dashed] (3.4, 1.8) -- (3.4, 4);
\draw[dashed] (0, 1.75) -- (3.4, 1.75);
\draw[dashed] (0, 2) -- (3.4, 2);
\draw[dashed] (2.6, 2) -- (2.6, 4);
\foreach \Point in {(0.5,0.06), (1.1,0.76), (1.6,1.96), (2.6,2.76), (3.3,3.46)}{
    \node at \Point {\textbullet};
}
\draw (3, 1.7) node {\textbullet};
\draw[thick,->] (2.2, 3.3) -- (2.5, 3.3) ;
\draw[thick,->] (1.1, 2.4) -- (1.1, 2.13) ;
\draw (3.8, 2.3) node[anchor=south] {\small $\Gamma_t$};
\draw[thick,->] (3.7, 2.85) -- (3.5, 3.25) ;
\draw[thick,->] (3.5, 2.6) -- (3, 2.65) ;
\draw[thick,->] (3.5, 2.5) -- (2.2, 2.13) ;
\draw[thick,->] (3.7, 2.4) -- (3.3, 2.1) ;
\draw (1.1, 2.4) node[anchor=south] {\small $R_1$};
\draw (1.8, 3) node[anchor=south] {\small $R_2$};
\draw (1.6, 0) node[anchor=north] {\small $r_{n-1}$} ;
\draw (2.6, 0) node[anchor=north] {\small $r_{n}$} ;
\draw (3.6, 0) node[anchor=north] {\small $O_{(n)}$} ;
\draw (0,1.55) node[anchor=east] {\small $L(D_{(n)})$};
\draw (0,2.1) node[anchor=east] {\small $L(\sigma_{(n-1)})$};
\end{tikzpicture}
\caption{Typical configurations of the sets $R_1$ and $R_2$. Depicted as well are the records $r_{n-1}$ and $r_n$, the boundaries $O_{(n)}$ and $D_{(n)}$ and the localisation set $\Gamma_t$.} \label{gtpic2}
\end{figure}

\begin{lemma}
\label{r1lem}
Suppose Assumption \ref{assumpt:sosv} holds. Then for each $k \in \mathbb{N}$ and $\varepsilon > 0$ there exists a constant $c >0$ and a sequence $(a_n)_{n\geq 1}$ satisfying $\sum_{n \in \mathbb{N}} a_n < \infty$ such that, as $n \to \infty$,
\[\mathbf{P}(|R_1|\geq k) < a_n + c \left( d(e^{n(1-\varepsilon)})  \hat{h}_n^2 \right)^k \]
eventually. In particular, $|R_1|\leq N-2$ eventually almost-surely.
\end{lemma}

\begin{proof}
By definition, the chaining window from $r_n$ is given by
\[  h_{(n)}\max \left\{  \frac{ t_n }{\sigma_{(n)} } , 1 \right \}  \le h_{(n)}^2\max \left\{ \frac{r_n S_{(n)}^-}{\sigma_{(n)}}, 1 \right\} .\]
Thus, by the union bound we find that
\[\mathbf{P}\left(|R_1|\geq 1|\: \sigma_{(n-1)}, \sigma_{(n)},r_n,S_{(n)}^-\right)\leq\frac{h^2_{(n)}}{L(\sigma_{(n-1)})}\max \left\{  \frac{r_n S_{(n)}^-}{\sigma_{(n)}},1\right\}.\]
Similarly, to bound the probability of there being at least $k$ sites in this region given that there is at least one, we can condition on the height of the first site, which is an i.i.d. copy of $\sigma_{(n)}$, and repeat the process. In particular,
\[\mathbf{P}\left(|R_1|\geq k|\: \sigma_{(n-1)},r_n,S_{(n)}^-\right)\leq \mathbf{E}\left(\prod_{i = 1, \ldots , k}\frac{h^2_{(n)}}{L(\sigma_{(n-1)})}\max \left\{  \frac{r_n S_{(n)}^-}{\sigma^i_{(n)}},1\right\} \bigg|\: \sigma_{(n-1)},r_n,S_{(n)}^-\right),
\]
where, under the conditioned law, $(\sigma^i_{(n)})_{i\geq 2}$ are i.i.d.\ copies of $\sigma_{(n)}$ (i.e. first exceedences of the level $\sigma_{(n-1)}$). Applying the independence and integrating out the $\sigma_{(n)}$ using Lemma~\ref{lem:ex}, it follows that, for some $c > 0$, as $\sigma_{(n-1)} \to \infty$ eventually
\begin{eqnarray}
\mathbf{P}\left(|R_1|\geq k|\: \sigma_{(n-1)},r_n,S_{(n)}^-\right)&\leq& \frac{(h_{(n)})^{2k}}{L(\sigma_{(n-1)})^k}\mathbf{E}\left(  \frac{r_n S_{(n)}^-}{\sigma_{(n)}} + 1 \bigg|\: \sigma_{(n-1)},r_n,S_{(n)}^-\right)^k\nonumber\\
& <  &
 \frac{(h_{(n)})^{2k}}{L(\sigma_{(n-1)})^k}  \left(\frac{ c \, r_n S_{(n)}^-   \, g(\sigma_{(n-1)})  }{\sigma_{(n-1)}}+1\right)^k\nonumber\\
& <  &  \frac{(2c)^{k} \, (h_{(n)})^{2k}}{L(\sigma_{(n-1)})^k}  \left(
\left(\frac{r_n S_{(n)}^- g(\sigma_{(n-1)})}{\sigma_{(n-1)}}\right)^k+1\right) .\label{condestr1}
\end{eqnarray}
Define the event
 \begin{equation}\label{bnest}
 \mathcal{B}_n := \left\{ \log L( \sigma_{(n-1)}) \in (1- \varepsilon, 1+\varepsilon)n \, , \ r_n < 2  L(\sigma_{(n-1)}) \log n \,  , \ S_{(n)}^- < N \sigma_{(n-1)} \right \} ,
 \end{equation}
and recall that, by Lemmas \ref{lem:probrecord} and \ref{anlem} and the union bound, this event satisfies $\mathbf{P}(\mathcal{B}^c_n) < a_n$ for some sequence satisfying $\sum_n a_n < \infty$. Further, by the definition of $\hat{h}_n$, note that on the event $\mathcal{B}_n$ we have $h_{(n)} < \hat{h}_n$. Consequently, by the choice of $\hat{h}_n$ in \eqref{eq:h4}, we have that, for some constant $c_1>0$, as $n \to \infty$ eventually
\[\mathbf{P}(|R_1|\geq k) < a_n + c_1 \, \left( d( e^{n(1-\varepsilon)}) \hat{h}_n^2 \right)^k \,
 \mathbf{E}\left(\left(\frac{r_n}{L(\sigma_{(n-1)})}\right)^k\right) . \]
Applying the moment estimate of Lemma \ref{lem:record}, this provides the probability estimate. The almost-sure claim follows by considering the definition of $N$ and the properties of $\hat{h}_n$ in~\eqref{eq:h4}, and then by applying a Borel-Cantelli argument.
\end{proof}

We now include into the analysis the set $R_2$. To increase our ability to exploit independence, we shall initially substitute the set $R_2$ for different set $\tilde{R}_2$, which contains $R_2$ with high probability but will be simpler to work with. To this end, recall the notation $S_t^\ell$ used to define the lower boundary $D_t$, and define
\[  \tilde{D}_{(n)} := \max \{ \ell \ge 0: S^\ell_{r_n \hat{h}^3_{n}} < \sigma_{(n-1)} / \hat{h}_{n} \} \]
and
\begin{eqnarray*}
\tilde{R}_2&:=&\left\{x:\:x\in(0, r_n \, \hat{h}^3_{n}),\: \tilde{D}_{(n)} \le \sigma_x < \sigma_{(n-1)} \right\} .
\end{eqnarray*}

Before analysing the cardinality of $\tilde{R}_2$, we make the link between the sets $R_2$ and $\tilde{R}_2$.

\begin{lemma}
\label{lem:r2r2} Suppose Assumption \ref{assumpt:sosv} holds. Then we have that
\[ \sum_{n \in \mathbb{N}} \mathbf{P}( R_2 \nsubseteq \tilde{R}_2 )  < \infty . \]
In particular, $R_2 \subseteq \tilde{R}_2$ eventually almost-surely.
\end{lemma}
\begin{proof}
Denote by $\mathcal{B}_n$ the event
\[ \mathcal{B}_n := \left\{ S_{(n)}^- < N \sigma_{(n-1)} , \:   |R_1| < N  ,\:h_{(n)}\leq\hat{h}_n \right\} .  \]
By considering the chaining that defines the outer boundary $O_t$, we observe that on $\mathcal{B}_n$ we know that $O_t <   r_n \hat{h}_n^3$. Given the respective definitions of $\tilde{D}_{(n)}$ and $D_{(n)}$, this in turn implies $\tilde{D}_{(n)} \le   D_{(n)}$ and hence $R_2 \subseteq \tilde{R}_2$. Hence we infer that
\[ \sum_{n \in \mathbb{N}} \mathbf{P}( R_2 \nsubseteq \tilde{R}_2 ) \le \sum_{n \in \mathbb{N}} \mathbf{P}(\mathcal{B}^c_n ),\]
and the result follows by Lemmas \ref{anlem} and \ref{r1lem} (as well as the observation from the proof of the latter lemma that $h_{(n)}\leq\hat{h}_n$ holds on an event of probability greater than $1-cn^2$).
\end{proof}

To complete our analysis of the cardinality of $\Gamma_t$ we bound $\mathbf{P}(|R_1| + |\tilde{R}_2| \ge N-1)$, from where the fact that $|\Gamma_t| \le N$ eventually almost-surely follows by a Borel-Cantelli argument.

\begin{lemma}
\label{r2lem} Suppose Assumption \ref{assumpt:sosv} holds. We have that
\[  \sum_{n \in \mathbb{N}} \mathbf{P}(|R_1| + |\tilde{R}_2|  \geq N-1) < \infty.  \]
In particular, $|R_1| + |\tilde{R}_2| \leq N-2$ eventually almost-surely.
\end{lemma}
\begin{proof}
First observe that
\begin{align}
\mathbf{P}(|R_1| + &  |\tilde{R}_2| \geq N-1) \nonumber\\
& = \sum_{k = 0}^{N-1} \mathbf{E} \left[ \mathbf{1}_{\{|R_1| \ge k \}} \mathbf{P}(|\tilde{R}_2| \ge N-1-k \:|\: \sigma_{(n-1)}, r_{n-1}, \{(x,\sigma_x):\:\sigma_x\geq \sigma_{(n-1)}\})  \right] .\label{bound}
\end{align}
To control the conditional probability in (\ref{bound}), note that $|\tilde{R}_2| \ge N-1-k$ implies that if we exclude the largest $N-2-k$ terms from sum
\[\sum_{i=1}^{r_n\hat{h}_n^3}\sigma_i\mathbf{1}_{\{\sigma_i< \sigma_{(n-1)}\}}=
\sum_{i=1}^{n-2}\sigma_{(i)}+\sum_{i=1}^{r_n\hat{h}_n^3}\sigma_i\mathbf{1}_{\{\sigma_i< \sigma_{(n-1)},\:\sigma_i\not\in\mathcal{R}\}},\]
then the result is still greater than $\sigma_{(n-1)}/\hat{h}_n$. Hence, by following the same argument as in the proof of the second part of Lemma \ref{anlem} we have that
\begin{eqnarray*}
\lefteqn{\mathbf{P}(|\tilde{R}_2| \ge N-1-k \:|\: \sigma_{(n-1)}, r_{n-1}, \{(x,\sigma_x):\:\sigma_x\geq \sigma_{(n-1)}\})}\\
&\leq &\mathbf{P} \left(\sum_{i=1}^{(n-2)-(N-k-2)}\sigma_{(i)}>\hat{h}_n^{-1}\sigma_{(n-1)} \:\vline\:\sigma_{(n-1)}\right)\\
&&+\sum_{l=1}^{N-k-2}\mathbf{P}\left(\sum_{i=1}^{(n-2)-(N-k-2-l)}\sigma_{(i)}\geq \hat{h}_n^{-1}\sigma_{(n-1)} \:\vline\:\sigma_{(n-1)}\right) F \left(r_n\hat{h}_n^3,\sigma_{(n-1)},l,\hat{h}_n^{-1}  \right).
\end{eqnarray*}
Note that, in contrast to the proof of Lemma \ref{anlem}, we have also included the traps in $(\sigma_{(n)}, r_n \hat{h}_{n}^3)$ into the analysis, but this causes no problem since under the relevant conditioning the terms $\sigma_i\mathbf{1}_{\{\sigma_i< \sigma_{(n-1)},\:\sigma_i\not\in\mathcal{R}\}}$ are either identically zero, or have the distribution $\sigma_0|\sigma_0<\sigma_{(n-1)}$. Now, conditioning on the event $\mathcal{B}_n$ defined at (\ref{bnest}) (whose probability is estimated in Lemmas \ref{lem:probrecord} and \ref{anlem}), applying the conditional estimate for the tail of $|R_1|$ of (\ref{condestr1}), and the estimate for $F$ from Corollary \ref{cor:sumlevel}, one deduces that
\begin{eqnarray*}
\lefteqn{\mathbf{P}(|R_1| +  |\tilde{R}_2| \geq N-1) }\\
& \leq & a_n +c \sum_{k = 0}^{N-1} (d(e^{n(1-\varepsilon)})\hat{h}_n^2)^k \times\left[\mathbf{P} \left(\sum_{i=1}^{(n-2)-(N-k-2)}\sigma_{(i)}>\hat{h}_n^{-1}\sigma_{(n-1)} \right)\right.\\
&&+\left.\sum_{l=1}^{N-k-2}\mathbf{P}\left(\sum_{i=1}^{(n-2)-(N-k-2-l)}\sigma_{(i)}\geq \hat{h}_n^{-1}\sigma_{(n-1)} \right)\hat{h}_n(d(e^{n(1-\varepsilon)})\hat{h}_n^3\log n)^l\right].
\end{eqnarray*}
for some sequence $(a_n)_{n\geq 1}$ with $\sum_n a_n < \infty$. To estimate the remaining probabilities, one can apply the argument of Lemma \ref{sumsofrecords}. The result is summable by the choice of $\hat{h}_n$ in \eqref{eq:h4}, which completes the proof of the first claim of the lemma. The almost-surely statement then follows from a Borel-Cantelli argument.
\end{proof}

From Lemmas  \ref{lem:r2r2} and \ref{r2lem}, we have the following corollary.

\begin{corollary}
\label{cor:N} Suppose Assumption \ref{assumpt:sosv} holds. We have that
\[ \sum_{n \in \mathbb{N}} \mathbf{P}( |R_1| + |\tilde{R}_2| \ge N-1 \mbox{ or } \ R_2 \nsubseteq \tilde{R}_2  )  < \infty \, .\]
In particular, $|R_1| + |{R}_2| \leq N-2$ eventually almost-surely.
\end{corollary}

Recalling that, by construction, $\limsup_{t\rightarrow\infty}|\Gamma_t|\leq 2+\limsup_{n\rightarrow\infty}(|R_1| + |{R}_2|)$, the previous result (together with Proposition \ref{prop:loc}) completes the proof of the first claim of Theorem \ref{thm:main1}. To complete the section, we point out a second easy corollary of the above, which confirms one implication of Theorem \ref{thm:main3} holds.

\begin{corollary} Suppose Assumption \ref{assumpt:sosv} holds. Assume $N =2$. Then, as $t \to \infty$,
\[ \Gamma_t \subseteq \mathcal{R} \]
eventually almost-surely.
\end{corollary}
\begin{proof} As noted above, $\Gamma_t\subseteq \{r_{n-1},r_n\}\cup R_1\cup R_2$ for $t\in[t_{n-1},t_n)$. However, we have
by Corollary \ref{cor:N} that $|R_1| + |{R}_2|\leq N-2=0 $ eventually almost-surely. Thus $\Gamma_t\subseteq \{r_{n-1},r_n\}\subseteq \mathcal{R}$ eventually almost-surely.
\end{proof}

\section{Most favoured site}
\label{sec:fav}

In this section we consider the most favoured site; that is, we complete the proof of Theorem~\ref{thm:main2}. We note that by the first part of Theorem \ref{thm:main1}, as was proved in the previous section, the most favoured site must have at least $1/N$ proportion of the probability mass at all sufficiently large times. Moreover, by Theorem \ref{thm:maincl}, also proved in the previous section, and combining with Fatou's lemma, one readily deduces that
\[  \limsup_{t \to \infty} \sup_{x \in \mathbb{Z}^+} P_\sigma(X_t = x) = 1 \, .  \]
Hence it is sufficient to show that there exist arbitrarily large times at which the probability mass of the BTM is evenly balanced across exactly $N$ sites. Proving this will also finish the proof of Theorem \ref{thm:main1}. Furthermore, to prove the converse direction of Theorem~\ref{thm:main3}, we just need to show that if $N \ge 3$ then additionally the $N$ sites referred to above are not all record traps.

We proceed in two steps. First, we establish the above `balanced localisation' result on the assumption that a certain favourable event $\mathcal{E}_{n}$ involving the trapping landscape holds infinitely often. Second, we prove that this favourable event does indeed hold infinitely often almost-surely; it is here that we will need to work under Assumption \ref{assumpt:g22}.

\subsection{Defining the favourable event}

In this section, we define the favourable event~$\mathcal{E}_n$. The definition of this involves a certain $n$-dependent collection of sites $(z_i)_{1 \le i \le N}$, (we drop the explicit dependence on $n$ for brevity). In particular, we fix $\varepsilon_0\in(0,1)$, define $z_1:=r_{n-1}$ and set, for $i=2,\dots, N$,
\[z_{i}=\min\left\{z>z_{i-1}:\:\sigma_z>(1-\varepsilon_0)\sigma_{(n-1)}\right\}.\]
We also introduce the notation
\[\Lambda_n:= L((1-\varepsilon_0)\sigma_{(n)}).\]
For $\varepsilon_1,\varepsilon_2,\varepsilon_3,\varepsilon_4,\varepsilon_5,\varepsilon_6,\varepsilon_7\in(0,1)$, satisfying $\varepsilon_1 < \varepsilon_2$, we then suppose $\mathcal{E}_n$ is defined to be the event
\begin{eqnarray*}
&&\left\{z_N=r_n,\:\frac{z_N-z_{N-1}}{\Lambda_{n-1}}\in(\varepsilon^{-1}_1,\varepsilon^{-1}_2),
\:\frac{z_{N-1}-z_1}{\Lambda_{n-1}}<\varepsilon_3^{-1},\:\frac{z_1}{\Lambda_{n-1}}<\varepsilon_4^{-1}\right\}\\
&&\cap\left\{\sum_{\substack{z<z_N:\\z\not\in\{z_1,\dots,z_{N-1}\}}}\sigma_z<\varepsilon_4\sigma_{(n-1)},\:\sigma_{(n)}>\varepsilon_5^{-1}\sigma_{(n-1)},
\:\sum_{z_N<z\leq z_N+\varepsilon_6^{-1}\Lambda_{n-1}}\sigma_z<\varepsilon_7\sigma_{(n-1)}
\right\}.
\end{eqnarray*}
Constraints on $(\varepsilon_i)_{i=0}^7$ will be imposed later as they become necessary for the argument. We will also allow these parameters to depend on $n$ where needed. (Actually we only need to do this for $\varepsilon_4$.) See Figure \ref{enpic} for a typical configuration on this event.

\begin{figure}[ht]
\begin{tikzpicture}
\draw[thick,->] (0,0)  -- (10,0) node[anchor=north west] {};
\draw[thick,->] (0,0)  -- (0,4.5) node[anchor=south] {\small $\sigma_z / \sigma_{(n-1)}$};
\draw[thick] (3, 0) -- (3, -0.1);
\draw (3, 0) node[anchor=north] {\small $z_1/ \Lambda_{n-1}$};
\draw[thick] (4.5, 0) -- (4.5, -0.1);
\draw (4.5, 0) node[anchor=north] {\small $z_2/ \Lambda_{n-1}$};
\draw[thick] (7.2, 0) -- (7.2, -0.1);
\draw (7.2, 0) node[anchor=north] {\small $z_{N-1}/ \Lambda_{n-1}$};
\draw[thick] (9, 0) -- (9, -0.1);
\draw (9, 0) node[anchor=north] {\small $z_N/ \Lambda_{n-1}$};
\draw[thick] (0, 0.4) -- (-0.1, 0.4);
\draw (0, 0.4) node[anchor=east] {\small $\min\{\varepsilon_4, \varepsilon_7\}$};
\draw[dashed] (0, 0.4) -- (10, 0.4);
\draw[thick] (0, 2) -- (-0.1, 2);
\draw (0, 2) node[anchor=east] {\small $1$};
\draw[thick] (0, 4) -- (-0.1, 4);
\draw (0, 4) node[anchor=east] {\small $\sigma_{(n)} / \sigma_{(n-1)}$};
\foreach \Point in {(1, 0.2), (5.3, 0.1), (9.6, 0.2), (3,2), (4.5,1.7), (5.2, 1.8), (6, 1.9), (7.2, 1.8), (9, 4)}{
    \node at \Point {\textbullet};
}
\draw[thick,<->] (0,2.5)  -- (1.5, 2.5) node[anchor=south] {\small $<  \varepsilon_4^{-1}$} -- (3, 2.5) ;
\draw[thick,<->] (3,3)  -- (5.1, 3) node[anchor=south] {\small $< \varepsilon_3^{-1}$} -- (7.2, 3) ;
\draw[thick,<->] (7.2,4.5)  -- (8.1, 4.5) node[anchor=south] {\small $\in ( \varepsilon_1^{-1}, \varepsilon_2^{-1})$} -- (9, 4.5) ;
\draw[thick,<->] (9, 0.6)  -- (9.6, 0.6) node[anchor=south] {\small $ \varepsilon_6^{-1}$} -- (10, 0.6) ;
\draw[thick,<->] (7.5,1.5)  -- (7.5, 1.75) node[anchor=west] {\small $< \varepsilon_0$} -- (7.5, 2) ;
\draw[thick,<->] (9.2,2)  -- (9.2, 3) node[anchor=west] {\small $> \varepsilon_5^{-1} $} -- (9.2, 4) ;
\end{tikzpicture}
\caption{A typical configuration on the favourable event $\mathcal{E}_n$, depicting the sites $\{z_i\}_{i=1}^7$ along with other records and near records of the sequence~$\sigma$.}\label{enpic}
\end{figure}
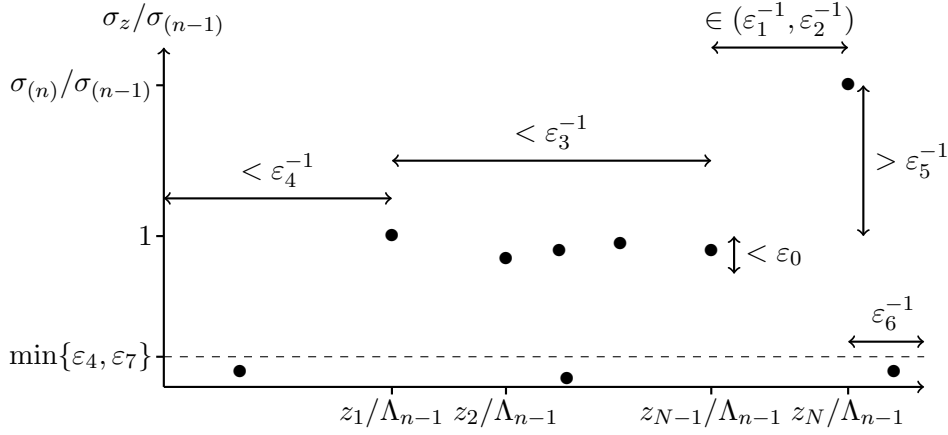

\subsection{Balanced localisation assuming the favourable event holds}

In this section, we prove the remaining part of Theorem \ref{thm:main2}, as well as the converse direction of Theorem~\ref{thm:main3}, under the following assumption.

\begin{assumption}
\label{assump}
The event $\mathcal{E}_n$ holds infinitely often almost-surely whenever  $(\varepsilon_i)_{i=0}^7$ satisfy $\varepsilon_1^{-1}<\varepsilon_2^{-1}$ and $\varepsilon_4=1/4\log n$.
\end{assumption}

To prove that balanced localisation occurs, the idea is to show that, under $\mathcal{E}_n$, the chain mixes quickly on $\{z_1,\dots,z_{N-1}\}$, meaning that on `short' time scales the mass is approximately evenly distributed on these sites, whereas on a suitably selected `long' time scale an appropriate amount of mass, approximately $1/N$, has seeped onto $z_N$. From this it is possible to conclude that, at this latter time, the process is approximately uniformly distributed over $\{z_1,\dots,z_{N}\}$.

We start by defining the long time scale. In particular, we choose $t_n$ to be the unique time such that
\[P_\sigma\left(\tau_{z_N}\leq t_n\right)=\frac{1}{N},\]
where $\tau_{x}$ is the hitting time of $x$ by $X$ (started from 0). (Note that for $x\geq 1$, $\tau_x$ has a continuous distribution with full support on $(0,\infty)$, and so $t_n$ is well-defined.) The following lemma is a ready consequence of this definition.

\begin{lemma}[Bounds on the long time scale]
\label{lem:tnbounds}
Suppose $\mathcal{E}_n$ holds, then
\[\frac{2(1-\varepsilon_0)}{\varepsilon_1N} <  \frac{t_n}{\Lambda_{(n-1)} \sigma_{(n-1)}} < 6 (\varepsilon_2^{-1}+\varepsilon_3^{-1})N^2  . \]
\end{lemma}
\begin{proof}
For the lower bound, applying the lower bound on hitting times in the first statement of Proposition \ref{prop:hittingtimelb} (with $a = x =0$, $b = z_N$ and $z = z_1$) and the definition of $\mathcal{E}_n$ we have that
\[P_{\sigma}\left(\tau_{z_N}\leq t\right) <  \frac{t}{2(z_N - z_1) \sigma_{z_1}}   <  \frac{t}{2(1-\varepsilon_0) \varepsilon_1^{-1} \Lambda_{(n-1)}\sigma_{(n-1)}}   <  \frac{t \varepsilon_1}{2(1-\varepsilon_0)\Lambda_{(n-1)}\sigma_{(n-1)}}    . \]
Given the definition of $t_n$, taking $ t= {2(1-\varepsilon_0)\Lambda_{(n-1)}\sigma_{(n-1)}}/{\varepsilon_1 N}$ in the above establishes the result. Similarly for the upper bound, applying the upper bound on hitting times in the second statement of Proposition \ref{prop:hittingtimeub} (with $a = x = 0$, $b = z_N$ and ${S} = \{z_1, \ldots, z_{N-1} \}$) and the definition of $\mathcal{E}_n$ we have that
\begin{eqnarray*}
\lefteqn{P_{\sigma} \left( \tau_{z_N} \le t \right)} \\
& >&  1- 2 t^{-1} \left( z_N \sum_{ \{z < z_N\} \setminus \{z_1, \ldots, z_{N-1}\} } \sigma_z  + (N-1)(z_N - z_1) \sigma_{{(n-1)} } \right)   \\
& > &1 - 2 t^{-1}  \left( (\varepsilon_2^{-1}+\varepsilon_3^{-1}+\varepsilon_4^{-1}) \varepsilon_4 \Lambda_{(n-1)} \sigma_{(n-1)} +  (\varepsilon_2^{-1}+\varepsilon_3^{-1})(N-1)\Lambda_{n-1}\sigma_{(n-1)}    \right) \\
& > &1 - 3 t^{-1} (\varepsilon_2^{-1}+\varepsilon_3^{-1})N\Lambda_{n-1}\sigma_{(n-1)}  .
\end{eqnarray*}
Taking $t = 6 (\varepsilon_2^{-1}+\varepsilon_3^{-1})N^2\Lambda_{n-1}\sigma_{(n-1)}$, we have $P_{\sigma}( \tau_{z_N} \le t) > 1 - \frac{1}{2N} > \frac{1}{N}$ (since $N \ge 2$), which establishes the result.
\end{proof}

We now bound the probability mass on the final site at times given by the long time scale, showing that it is very nearly $1/N$.

\begin{lemma}[Probability mass on the final site]
\label{rnbal} Suppose $\mathcal{E}_n$ holds, then
\[P_\sigma\left(X_{t_n}=z_N\right) > \frac{1}{N}\left(1-
\varepsilon_5(\varepsilon_4+\varepsilon_7)-
6\varepsilon_5\max\{\varepsilon_1,\varepsilon_6\} (\varepsilon_2^{-1}+\varepsilon_3^{-1})N^2
\right).\]
\end{lemma}

\begin{proof}
By the Markov property and the definition of $t_n$, we have that
\begin{eqnarray*}
P_\sigma\left(X_{t_n}=z_N\right)&=&P_\sigma\left(X_{t_n}=z_N,\:\tau_{z_N}\leq t_n\right)\\
&\geq &\frac{1}{N}\min_{t\leq t_n}P_\sigma\left(X_t=z_N\:|\:X_0=z_N\right).
\end{eqnarray*}
To estimate the probability here, we consider the inhomogeneous CTRW on
\[ \Omega_n:= [z_N-\varepsilon_{1}^{-1}\Lambda_{n-1}, z_N+\varepsilon_{6}^{-1}\Lambda_{n-1}] \cap \mathbb{Z^+}  \]
in the trapping landscape  $(\sigma_x)_{x\in \Omega_n}$  (see Section \ref{sec:prelim} for the definition of this Markov chain).

Note that on $\mathcal{E}_n$ we have $\Omega_n \subseteq \mathbb{Z}^+$, and hence, in particular, if $X$ and $X^n$ are both started from $z_N$, then their distributions are the same up to the hitting time of the endpoints of $\Omega_n$. Denoting the latter stopping time by $\tau$, we thus obtain
\begin{eqnarray*}
P_\sigma\left(X_{t}=z_N\:|\:x_0=z_N\right)&\geq & P_\sigma\left(X_{t}=z_N,\:\tau>t\:|\:X_0=z_N\right)\\
&=&P^n_{z_N}\left(X^n_{t}=z_N,\:\tau>t\right)\\
&\geq & 1-P^n_{z_N}\left(X^n_{t}\neq z_N\right)-P^n_{z_N}\left(\tau\leq t\right),
\end{eqnarray*}
where ${P}^n_x$ is the law of ${X}^n$ started from $x$.

To bound the first term, applying the localisation result in Proposition \ref{prop:local} (with $z = z_N$ and ${S} = \Omega_n \setminus \{z_N\}$) and the definition of $\mathcal{E}_n$ yields that, for $t \ge 0$,
\[P^n_{z_N}\left(X^n_{t}\neq z_N\right) \leq \frac{  \sum_{z \in \Omega_n \backslash\{z_N\} } \sigma_z }{\sigma_{z_N}} < \varepsilon_5(\varepsilon_4+\varepsilon_7) .\]
To bound the second term, applying the lower bound on hitting times in the second statement of Proposition \ref{prop:hittingtimelb} (with $x = z_N$) and the definition of $\mathcal{E}_n$ yields that, for $t \ge 0$,
\[P^n_{z_N}\left(\tau\leq t\right) <  \frac{ t\varepsilon_5 }  {\min\{  \varepsilon_1^{-1}, \varepsilon_6^{-1} \} \Lambda_{n-1} \sigma_{z_N} } <  \frac{  t \varepsilon_5 \max\{\varepsilon_1,\varepsilon_6\} }{ \Lambda_{n-1}\sigma_{(n-1)} }. \]
Combining with the upper bound on $t_n$ in Lemma \ref{lem:tnbounds}, we have that
\[P^n_{z_N}\left(\tau\leq t_n \right)  <   6 \varepsilon_5 \max\{\varepsilon_1,\varepsilon_6\}( \varepsilon_2^{-1} + \varepsilon_3^{-1}) N^2 . \]
and the result follows.
\end{proof}

We proceed now to investigate the short time scale. To this end, let $\tilde{X}^n$ be the inhomogeneous CTRW on $[0, z_N] \cap \mathbb{Z}$ in the trapping landscape $(\sigma^n_x)_{x \in \mathbb{Z}^+}$ consisting of a copy of $\sigma$ but with $\sigma_{z_N}$ replaced by $\varepsilon_4\sigma_{(n-1)}$. In particular, the processes $X$ and $\tilde{X}^n$ have the same distribution up to hitting $z_N$. For $\varepsilon\in(0,1)$, we define
\[t_{\rm mix}^n(\varepsilon):=\inf\left\{t\geq 0:\:\max_{i=1,\dots,N-1}\sum_{0\leq y\leq z_N}\left|\tilde{P}^n_{z_i}\left(\tilde{X}^n_t=y\right)-\tilde{\pi}^n(y)\right|\leq \varepsilon\right\},\]
where $\tilde{P}^n_x$ is the law of $\tilde{X}^n$ started from $x$, and $\tilde{\pi}^n$ is its invariant probability measure. This is a version of the ($\varepsilon$-)mixing time of $\tilde{X}^n$, and will provide the short time scale for our argument. Note in particular that the maximum is only taken over starting points from $(z_i)_{i=1}^{N-1}$ rather than the entire interval $[0, z_N] \cap \mathbb{Z}$, as would be the case in the usual definition of a mixing time. In this setting, for suitable values of $\varepsilon$ this results in a quantity of a much lower order, since the majority of the mass is explored quickly when $\tilde{X}^n$ is started from one of the sites in the collection $(z_i)_{i=1}^{N-1}$. (Conversely, if the process is started from a site close to $z_N$, it would take a relatively long time to find the vertices $(z_i)_{i=1}^{N-1}$.) The following lemma provides a key estimate on this quantity. The proof is based on an argument for upper bounding mixing times presented in \cite{Aldous82} that involves considering a stopping time at which the random walk hits a stationary random vertex.

\begin{lemma}[Mixing on short time scales]
\label{mixest}  Suppose $\mathcal{E}_n$ holds, and that $2\varepsilon_4\leq \varepsilon_0\leq N^{-1}$. Then
\[ \frac{ t_{\rm mix}^n\left(8{\varepsilon}_0^{1/2}\right) }{\Lambda_{n-1}\sigma_{(n-1)}}  <  \frac{N}{\varepsilon_0^2\varepsilon_3}.\]
\end{lemma}

\begin{proof} We start by defining a randomised stopping time $T$, which will be the time taken to hit an almost stationary random vertex in $\{z_i:\:i=1,\dots,N-1\}$. In particular, let $Z$ be a random vertex in $\{z_i:\:i=1,\dots,N-1\}$ that is independent of $\tilde{X}^n$ and satisfies
\[\mathbf{P}\left(Z=z_i\right)=\frac{\tilde{\pi}^n(z_i)}{\sum_{j=1}^{N-1}\tilde{\pi}^n(z_j)},\]
and set
\[T:=\inf\{t\geq 0:\:\tilde{X}^n_t=Z\}.\]
Clearly $\tilde{P}^n_x(X_T=z_i)=\mathbf{P}(Z=z_i)$ for each $x\in \{0,\dots,z_N\}$, $i=1,\dots,N-1$. (For simplicity of notation, we suppose the law $\tilde{P}^n_x$ is the joint law of $\tilde{X}^n$ and $Z$.) Moreover, we have that
\[  \tilde{P}^n_{z_i}(T > t) \le \max\left\{  \tilde{P}^n_{z_1}(\tau_{z_{N-1}} > t) \, , \ \tilde{P}^n_{z_{N-1}}(\tau_{z_{1}} > t)  \right\}  . \]
Applying the upper bound on hitting times in the first statement of Proposition \ref{prop:hittingtimeub} (first with $a = 0, b = z_{N-1}$ and $x = z_1$, and then with $a = z_N, b = z_1$ and $x = z_{N-1}$, by symmetry) and the definition of $\mathcal{E}_n$, yields that (recalling that $\sigma^n_{z_N} = \varepsilon_4 \sigma_{(n-1)}$)
\begin{align}
\label{tbound}
\tilde{P}^n_{z_i}(T > t) & \le    2t^{-1} (z_{N-1} - z_1)  \bigg( \sum_{0 \le z \le z_{N}-1} \sigma_z  + \varepsilon_4 \sigma_{(n-1)} \bigg) \\
\nonumber & <  2t^{-1} \varepsilon_3^{-1} \Lambda_{n-1} \sigma_{(n-1)} \left( N - 1 +  2\varepsilon_4  \right) < 4 N t^{-1}  \varepsilon_3^{-1} \Lambda_{n-1} \sigma_{(n-1)}.
\end{align}

The remainder of the proof closely follows \cite{Aldous82}. Specifically, for $i=1,\dots,N-1$ and $t,t_0>0$, we can write
\begin{eqnarray}
\sum_{0\leq y\leq z_N}\left|\tilde{P}^n_{z_i}\left(\tilde{X}^n_t=y\right)-\tilde{\pi}^n(y)\right|
& \! \! \! \! \leq& \! \! \! \! \! \sum_{0\leq y\leq z_N}\left|\tilde{P}^n_{z_i}\left(\tilde{X}^n_t=y\right)-\tilde{P}^n_{z_i}\left(\tilde{X}^n_t=y,\:T\leq t_0\right)\right|\label{bbb}\\
&&+\sum_{0\leq y\leq z_N}\left|\tilde{P}^n_{z_i}\left(\tilde{X}^n_t=y,\:T\leq t_0\right)-\tilde{\pi}^n(y)\right|.\nonumber
\end{eqnarray}
Let us denote the two sums on the right-hand side by $S_1$ and $S_2$. To deal with the first of these, we simply note that
\begin{equation}\label{bbb1}
S_1=\tilde{P}^n_{z_i}\left(T> t_0\right) <  \frac{4N\varepsilon_3^{-1}\Lambda_{n-1}\sigma_{(n-1)}}{t_0},
\end{equation}
where we have applied (\ref{tbound}) to deduce the inequality. For the second term, we start by applying Cauchy-Schwarz to deduce
\begin{eqnarray}
S_2^2&\leq& \sum_{0\leq y\leq z_N}\frac{1}{\tilde{\pi}^n(y)}\left(\tilde{P}^n_{z_i}\left(\tilde{X}^n_t=y,\:T\leq t_0\right)-\tilde{\pi}^n(y)\right)^2\nonumber\\
&=&\sum_{0\leq y\leq z_N}\frac{1}{\tilde{\pi}^n(y)}\tilde{P}^n_{z_i}\left(\tilde{X}^n_t=y,\:T\leq t_0\right)^2-2\tilde{P}^n_{z_i}\left(T\leq t_0\right)+1.\label{s2}
\end{eqnarray}
Now, define a measure $\nu$ on $\{0,\dots,z_N\}\times [0,t_0]$ by setting
\[\nu\left(\cdot,\cdot\right):=\tilde{P}^n_{z_i}\left(T\leq t_0,\:\left(\tilde{X}^n_T,T\right)\in\left(\cdot,\cdot\right)\right),\]
and a function $f$ on $\{0,\dots,z_N\}\times \mathbb{R}^+$ by
\[f(x,s):=\tilde{P}^n_{z_i}\left(T\leq t_0,\:\tilde{X}^n_{t_0+s/2}=x\right)\equiv\int_{\{0,\dots,z_N\}\times [0,t_0]}\tilde{P}^n_{y}\left(\tilde{X}^n_{t_0+s/2-r}=x\right)\nu(dy,dr).\]
By the definition of $f$ and reversibility, it is possible to deduce that
\begin{eqnarray*}
\lefteqn{\sum_{0\leq y\leq z_N}\frac{1}{\tilde{\pi}^n(y)}f(y,s)^2}\\
&=&\int\int\sum_{0\leq y\leq z_N}\frac{1}{\tilde{\pi}^n(y)}\tilde{P}^n_{y_1}\left(\tilde{X}^n_{t_0+s/2-r_1}=y\right)\tilde{P}^n_{y_2}\left(\tilde{X}^n_{t_0+s/2-r_2}=y\right)\nu(dy_1,dr_1)\nu(dy_2,dr_2)\\
&=&\int\int\frac{1}{\tilde{\pi}^n(y_2)}\tilde{P}^n_{y_1}\left(\tilde{X}^n_{2t_0+s-r_1-r_2}=y_2\right)\nu(dy_1,dr_1)\nu(dy_2,dr_2),
\end{eqnarray*}
where each of the integrals above is over $\{0,\dots,z_N\}\times [0,t_0]$. Hence, for any $s_0$,
\begin{eqnarray*}
\lefteqn{\frac{1}{s_0}\int_0^{s_0}\sum_{0\leq y\leq z_N}\frac{1}{\tilde{\pi}^n(y)}f(y,s)^2ds}\\
&\leq& \frac{1}{s_0}\int\int\frac{1}{\tilde{\pi}^n(y_2)}\int_0^{2t_0+s_0}\tilde{P}^n_{y_1}\left(\tilde{X}^n_{s}=y_2\right)ds\nu(dy_1,dr_1)\nu(dy_2,dr_2)\\
&\leq &\frac{1}{s_0}\sum_{j_1,j_2=1}^{N-1}
\frac{1}{\tilde{\pi}^n(z_{j_2})}\int_0^{2t_0+s_0}\tilde{P}^n_{z_{j_1}}\left(\tilde{X}^n_{s}=z_{j_2}\right)ds
\frac{\tilde{\pi}^n(z_{j_1})\tilde{\pi}^n(z_{j_2})}{\left(\sum_{k=1}^{N-1}\tilde{\pi}^n(z_k)\right)^2},
\end{eqnarray*}
where for the second inequality we note that the inner integral does not depend on $r_1$ or $r_2$, and apply the definition of $\nu$ and the stopping time $T$. To estimate the right-hand side here, we note that $\tilde{\pi}^n(x)$ is proportional to $\sigma^n_x$, from which it is elementary to check that (recalling that $\sigma^n_{z_N} = \varepsilon_4 \sigma_{(n-1)}$)
\begin{equation}\label{pinlower}
\tilde{\pi}^n\left(\{z_k:\:k=1,\dots,N-1\}\right) >  \frac{(N-1)(1-\varepsilon_0)}{N-1+2\varepsilon_4}.
\end{equation}
It follows that
\[\frac{1}{s_0}\int_0^{s_0}\sum_{0\leq y\leq z_N}\frac{1}{\tilde{\pi}^n(y)}f(y,s)^2ds < \frac{2t_0+s_0}{s_0}\times\frac{N-1+2\varepsilon_4}{(N-1)(1-\varepsilon_0)}.\]
For $\varepsilon_0,\varepsilon_4$ satisfying the assumptions of the lemma, and $s_0:=t_0/2\varepsilon_0$, it is straightforward to check that this implies
\[\frac{1}{s_0}\int_0^{s_0}\sum_{0\leq y\leq z_N}\frac{1}{\tilde{\pi}^n(y)}f(y,s)^2ds <  \left(1+4\varepsilon_0\right)^2 < 1+16\varepsilon_0.\]
In particular, there must exist an $s\leq s_0$ such that $\sum_{0\leq y\leq z_N}\frac{1}{\tilde{\pi}^n(y)}f(y,s)^2 < 1+16\varepsilon_0$, and, returning to (\ref{s2}),
\begin{equation}\label{bbb2}
S_2^2 <  16\varepsilon_0+2\tilde{P}^n_{z_i}\left(T> t_0\right)
\end{equation}
for some $t\leq t_0+s_0/2$. Since the left-hand side of (\ref{bbb}) is decreasing, if we choose $t_0:=2N\Lambda_{n-1}\sigma_{(n-1)}/\varepsilon_0\varepsilon_3$, then, by (\ref{bbb1}) and \eqref{bbb2},
\[\sum_{0\leq y\leq z_N}\left|\tilde{P}^n_{z_i}\left(\tilde{X}^n_{t_0+s_0/2}=y\right)-\tilde{\pi}^n(y)\right|\leq S_1+S_2  < 8\varepsilon^{1/2}_0.\]
The result follows.
\end{proof}

We are now ready to put the pieces together to establish that under $\mathcal{E}_n$ the process $X$ is approximately balanced on the sites $\{z_1,\dots,z_{N-1}\}$ at time $t_n$. Since the result for $i=N$ was already established as Lemma \ref{rnbal}, this will be enough to conclude the relevant part of Theorem \ref{thm:main2}.

\begin{proposition}[Balancing of the probability mass]
\label{balance} Suppose $\mathcal{E}_n$ holds, that $2\varepsilon_4\leq \varepsilon_0\leq N^{-1}$ and also that $N^2\varepsilon_1\leq (1-\varepsilon_0)\varepsilon_0^2\varepsilon_3$. Then, there exists a constant $c$ depending only on $N$ such that, for $i=1,\dots,{N-1}$,
\[P_\sigma\left(X_{t_n}=z_i\right) >  \frac{1}{N}\left(1-c\left(\varepsilon_0^{1/2}+\frac{\varepsilon_1}{\varepsilon_0^2\varepsilon_3}\right)\right).\]
\end{proposition}

\begin{proof} We first note that under the assumptions on the $(\varepsilon_i)_{i=0}^7$, the lower bound in Lemma \ref{lem:tnbounds} and Lemma \ref{mixest} imply that $t_n >  2t_{\rm mix}^n(\tilde{\varepsilon})$, where $\tilde{\varepsilon}:=8\varepsilon_0^{1/2}$. So, we have that $s_n:=t_n-t_{\rm mix}^n(\tilde{\varepsilon}) >  t_{\rm mix}^n(\tilde{\varepsilon})>0$, and we can apply the Markov property at time $s_n$ to obtain, for $i\in\{1,\dots, N-1\}$,
\begin{eqnarray}
P_\sigma\left(X_{t_n}=z_i\right)&\geq &
P_\sigma\left(X_{t_n}=z_i,\:\tau_{z_N}>t_n\right)\nonumber\\
&\geq&\sum_{j=1}^{N-1}P_\sigma\left({X}_{s_n}=z_j,\:\tau_{z_N}>s_n\right)\tilde{P}^n_{z_j}\left(\tilde{X}^n_{t_{\rm mix}^n(\tilde{\varepsilon})}=z_i,\:\tau_{z_N}>t_{\rm mix}^n(\tilde{\varepsilon})\right).\label{twoterms}
\end{eqnarray}
For the second probability in the above expression, we have that
\begin{eqnarray*}
\lefteqn{\tilde{P}^n_{z_j}\left(\tilde{X}^n_{t_{\rm mix}^n(\tilde{\varepsilon})}=z_i,\:\tau_{z_N}>t_{\rm mix}^n(\tilde{\varepsilon})\right)}\\
&\geq& \tilde{P}^n_{z_j}\left(\tilde{X}^n_{t_{\rm mix}^n(\tilde{\varepsilon})}=z_i\right)-\tilde{P}^n_{z_j}\left(\tau_{z_N}\leq t_{\rm mix}^n(\tilde{\varepsilon})\right)\\
& > & \tilde{\pi}^n(z_i)-\tilde{\varepsilon}-\frac{ t_{\rm mix}^n(\tilde{\varepsilon})}{(1-\varepsilon_0)\varepsilon_1^{-1}\Lambda_{(n-1)}\sigma_{(n-1)}},
\end{eqnarray*}
where, to deduce the final inequality, we have applied the definition of the mixing time, and bounded $\tilde{P}^n_{z_j}(\tau_{z_N}\leq t_{\rm mix}^n(\tilde{\varepsilon}))$ using the lower bound on hitting times in the first statement of Proposition \ref{prop:hittingtimelb} (with $a = 0$, $b = z_N$ and $x = z = z_j$). Plugging in the definition of $\tilde{\varepsilon}$, noting the bound of Lemma \ref{mixest}, and estimating the measure similarly to (\ref{pinlower}), we thus find that $\tilde{P}^n_{z_j}(\tilde{X}^n_{t_{\rm mix}^n(\tilde{\varepsilon})}=z_i,\:\tau_{z_N}>t_{\rm mix}^n(\tilde{\varepsilon}))$ is bounded below by
\begin{equation}\label{lowerfirst}
\frac{1}{N-1}-\frac{3\varepsilon_0}{(1-\varepsilon_0)}-8\varepsilon_0^{1/2}-
\frac{N\varepsilon_1}{(1-\varepsilon_0)\varepsilon_0^2\varepsilon_3}.
\end{equation}
For the remaining part of the sum at (\ref{twoterms}), we have that
\begin{eqnarray*}
\sum_{j=1}^{N-1}P_\sigma\left({X}_{s_n}=z_j,\:\tau_{z_N}>s_n\right)&\geq&
P_\sigma\left(\tau_{z_N}>s_n\right)-\tilde{P}^n_0\left(\tilde{X}_{s_n}\not\in\{z_j:\:j=1,\dots,{N-1}\}\right).
\end{eqnarray*}
Clearly $P_\sigma(\tau_{z_N}>s_n)>(N-1)/N$, by the definition of $t_n$. Furthermore
\begin{eqnarray*}
\lefteqn{\tilde{P}^n_0\left(\tilde{X}_{s_n}\not\in\{z_j:\:j=1,\dots,{N-1}\}\right)}\\
&\leq &P_\sigma\left(\tau_{z_1}>s_n\right)+\sup_{t\geq 0}\tilde{P}^n_{z_1}\left(\tilde{X}_t \not\in\{z_j:\:j=1,\dots,{N-1}\}\right)\\
& < & N\varepsilon_1+ \frac{\varepsilon_4}{(N-1)(1-\varepsilon_0)},
\end{eqnarray*}
where in the second inequality we used the upper bound on hitting times in the first statement of Proposition \ref{prop:hittingtimeub} (with $a = x = 0$ and $b = z_1$), the lower bound on $s_n >t_n/2$ of Lemma \ref{lem:tnbounds}, and the localisation result in Proposition \ref{prop:local} (with ${S} = [0, z_N] \cap \mathbb{Z}^+ \setminus \{z_1, \ldots, z_N \}$).

In particular, we conclude that
\[\sum_{j=1}^{N-1}P_\sigma\left({X}_{s_n}=z_j,\:\tau_{r_n}>s_n\right) >
\frac{N-1}{N}- N\varepsilon_1-\frac{\varepsilon_4}{(N-1)(1-\varepsilon_0)}.\]
Putting this together with (\ref{lowerfirst}), we obtain the result.
\end{proof}

\begin{corollary}[Completion of the proof of Theorem \ref{thm:main2}]
\label{cor:main2}
Suppose Assumption \ref{assump} holds. Then
\[\liminf_{t\rightarrow\infty}\sup_{x\in\mathbb{Z}^+}P_\sigma\left(X_t=x\right)=\frac{1}{N} \qquad \mathbf{P}\text{-almost-surely.}\]
\end{corollary}

\begin{proof} Fix $\varepsilon\in(0,1)$. By Assumption \ref{assump}, we can suppose that almost-surely there exists an infinite sequence $(n_i)_{i\geq 1}$ such that $\cap_{i\geq1}\mathcal{E}_{n_i}$ holds for $\varepsilon_0=\varepsilon_5=\varepsilon^2$, $\varepsilon_1=\varepsilon_6=\varepsilon^6$, $\varepsilon_3=\varepsilon_7=\varepsilon$, $\varepsilon_4=\frac{1}{4\log n_i}$ and $\varepsilon_2=\varepsilon^7$. Now, by Lemma \ref{rnbal} and Proposition \ref{balance}, on $\cap_{i\geq1}\mathcal{E}_{n_i}$ we have that
\[\limsup_{i\rightarrow \infty}\sup_{z\in\mathbb{Z}^+}P_\sigma\left(X_{t_{n_i}}=z\right)\leq \frac1N\left(1+c\varepsilon\right),\]
where $c$ is some deterministic constant. Thus to complete the proof, it will suffice to show that on $\cap_{i\geq1}\mathcal{E}_{n_i}$ we also have $t_{n_i}\rightarrow\infty$ almost-surely. By Lemmas \ref{lem:probrecord} and \ref{lem:tnbounds}, we have that on $\cap_{i\geq1}\mathcal{E}_{n_i}$,
$t_{n_i} >  N^{-1}\varepsilon_1^{-1}\Lambda_{n_i-1}\sigma_{(n_i-1)}\rightarrow\infty$ almost-surely, as desired.
\end{proof}

\begin{corollary}[Completion of the proof of Theorem \ref{thm:main3}]\label{cor:main3} Suppose Assumption \ref{assump} holds and that $N \ge 3$. Then
\[ \limsup_{t \to \infty} P_\sigma(X_t\not \in \mathcal{R})
\ge  \frac{N-2}{N}  \qquad  \mathbf{P}\text{-almost-surely.} \]
\end{corollary}
\begin{proof}
As in the proof of Corollary \ref{cor:main2}, by Assumption \ref{assump}, Proposition \ref{balance} and Lemmas \ref{lem:probrecord} and \ref{lem:tnbounds}, there exists a sequence $(n_i)_{i \ge i}$ such that $\cap_{i\geq1}\mathcal{E}_{n_i}$ holds for a certain choice of $(\varepsilon_i)_{0 \le i \le 7}$, and such that, on $\cap_{i\geq1}\mathcal{E}_{n_i}$, both
\[  \limsup_{i \to \infty}  P_\sigma \left(  X_{t_{n_i}} \in \{z_2, \ldots, z_{N-1} \}  \right)  \ge  \frac{N-2}{N} \]
and $t_{n_i} \to \infty$ hold almost-surely. To complete the proof, note simply that on the event $\mathcal{E}_n$ the sites $z_2, \ldots , z_{N-1}$ are not contained in the record traps $\mathcal{R}$ by definition.
\end{proof}

\subsection{The favourable event occurs infinitely often}

In this section, we establish that the event $\mathcal{E}_n$ occurs infinitely often; throughout we shall work under Assumptions \ref{assumpt:sosv} and~\ref{assumpt:g22}. Our proof breaks down into two parts. We start by considering the part involving the sum over vertices $z<z_N$. In particular, note that, for any $\varepsilon_0\in(0,1)$,
\[\left\{z_N=r_n,\:\sum_{\substack{z<z_N:\\z\not\in\{z_1,\dots,z_{N-1}\}}}\sigma_z>\varepsilon_4\sigma_{(n-1)}\right\}\subseteq
\left\{S_{r_n-1}^{(N)}>\varepsilon_4\sigma_{(n-1)}\right\},\]
and we recall from Lemma \ref{anlem} that if $\varepsilon_4:=1/4\log n$ then the right-hand side only occurs for finitely many $n$ almost-surely (recall that we are working under Assumptions \ref{assumpt:sosv} and \ref{assumpt:g22}). Moreover, for the same choice of $\varepsilon_4$, we have from Lemma \ref{lem:probrecord} and the slow-variation of $L$ that, for any $\varepsilon_0\in(0,1)$, $z_1=r_{n-1}<2\log n L(\sigma_{(n-2)}) <\varepsilon_4^{-1}\Lambda_{n-1}$ eventually almost-surely. Hence to show that $\mathcal{E}_n$ occurs infinitely often with $\varepsilon_4=1/4\log n$, it will be enough to show that
\begin{eqnarray*}
\tilde{\mathcal{E}}_n&:=&\left\{z_N=r_n,\:\frac{z_N-z_{N-1}}{\Lambda_{n-1}}\in(\varepsilon^{-1}_1,\varepsilon^{-1}_2),
\:\frac{z_{N-1}-z_1}{\Lambda_{n-1}}<\varepsilon_3^{-1}\right\}\\
&&\cap\left\{\sigma_{(n)}>\varepsilon_5^{-1}\sigma_{(n-1)},
\:\sum_{z_N<z<z_N+\varepsilon_6^{-1}\Lambda_{n-1}}\sigma_z<\varepsilon_7\sigma_{(n-1)}
\right\}.
\end{eqnarray*}
holds infinitely often. Together with the conditional Borel-Cantelli lemma, the following lemma establishes that this is indeed the case.

\begin{proposition} Suppose $(\varepsilon_i)_{i = 0}^7$ are such that $\varepsilon_i \in (0, 1)$, $\varepsilon_1^{-1}<\varepsilon_{2}^{-1}$ and $\varepsilon_4:=1/4\log n$. Let $\mathcal{F}_n$ denote the filtration generated by $\{ \sigma_z: z \leq r_{n+1} \}$. Then $\tilde{\mathcal{E}}_n\in\mathcal{F}_n$, and, under Assumptions \ref{assumpt:sosv} and \ref{assumpt:g22},
\[ \sum_n \mathbf{P}\left( \tilde{\mathcal{E}}_{2(n+1)}  | \mathcal{F}_{2n} \right) = \infty  \]
almost-surely.
\end{proposition}

\begin{proof} It is clear that
\[\left\{z_N=r_n,\:\frac{z_N-z_{N-1}}{\Lambda_{n-1}}\in(\varepsilon^{-1}_1,\varepsilon_2^{-1}),
\:\frac{z_{N-1}-z_1}{\Lambda_{n-1}}<\varepsilon_3^{-1},\:\sigma_{(n)}>\varepsilon_5^{-1}\sigma_{(n-1)}\right\}\in \mathcal{F}_n.\]
Moreover, on the event
\[\left\{\sum_{z_N<z<z_N+\varepsilon_6^{-1}\Lambda_{n-1}}\sigma_z<\varepsilon_7\sigma_{(n-1)}\right\},\]
we have that $\sigma_z<\varepsilon_7{\sigma_{(n-1)}}<\sigma_{(n)}$ for all $z\in \{z_N+1,\dots,z_N+\varepsilon_6^{-1}\Lambda_{n-1}-1\}$, and so it must be the case that $r_{n+1}\geq z_N+\varepsilon_6^{-1}\Lambda_{n-1}$. In particular, the sum is $\mathcal{F}_n$-measurable. Thus we conclude that $\tilde{\mathcal{E}}_{n}\in\mathcal{F}_n$, as desired.

For the remainder of the proof, it is useful to note that $\tilde{\mathcal{E}}_{n}$ contains the following event:
\[\left\{\frac{z_{i+1}-z_i}{\Lambda_{n-1}}<\frac{\varepsilon_3^{-1}}{N-2},\:i=1,\dots,N-2\right\}\cap
\left\{\sigma_{z_i}\leq \sigma_{(n-1)},\:i=2,\dots,N-1\right\}\]
\[\cap\left\{\sigma_{(n)}>\varepsilon_5^{-1}\sigma_{(n-1)}\right\}\cap
\left\{\frac{z_N-z_{N-1}}{\Lambda_{n-1}}\in(\varepsilon_1^{-1},\varepsilon_2^{-1})\right\}\cap\left\{\sum_{z_N<z<z_N+\varepsilon_6^{-1}\Lambda_{n-1}} \sigma_z<\varepsilon_7\sigma_{(n-1)}
\right\}.\]
Hence, it is easy to see using the independence properties of $(\sigma_x)_{x\geq 0}$ that
\[\mathbf{P}\left( \tilde{\mathcal{E}}_{2(n+1)}  | \mathcal{F}_{2n} \right)\geq p_1^{N-2}p_2^{N-2}p_3p_4p_5,\]
where
\begin{eqnarray*}
p_1&=&\mathbf{P}\left(z_{2}-z_1<\frac{\Lambda_{2n+1}\varepsilon_3^{-1}}{N-2}\:\vline\:\sigma_{(2n+1)}\right),\\
p_2&=&\mathbf{P}\left(\sigma_{z_2}\leq \sigma_{(2n+1)}\:\vline\:\sigma_{(2n+1)}\right),\\
p_3&=&\mathbf{P}\left(\sigma_{(2n+2)}>\varepsilon_5^{-1}\sigma_{(2n+1)}
\:\vline\:\sigma_{(2n+1)}\right),\\
p_4&=&\mathbf{P}\left(\frac{z_N-z_{N-1}}{\Lambda_{2n+1}}\in(\varepsilon_1^{-1},\varepsilon_2^{-1})\:\vline\:\sigma_{(2n+1)}\right),\\
p_5&=&\mathbf{P}\left(\sum_{z_N<z<z_N+\varepsilon_6^{-1}\Lambda_{2n+1}}\sigma_z<\varepsilon_7\sigma_{(2n+1)}\:\vline\:\sigma_{(2n+1)}\right).
\end{eqnarray*}
(Note that the indices relate to $\tilde{\mathcal{E}}_{2(n+1)}$, and so $z_1=r_{2n+1}$ and $z_N=r_{2n+2}$.) In particular, it is straightforward to check that, almost-surely,
\[p_1= 1-\left(1-\frac{1}{\Lambda_{2n+1}}\right)^{\frac{\varepsilon_3^{-1}\Lambda_{2n+1}}{N-2}}\geq 1-e^{-\frac{\varepsilon_3^{-1}}{N-2}}>0,\]
\[p_2=1-\frac{L((1-\varepsilon_0)\sigma_{(2n+1)})}{L(\sigma_{(2n+1)})}
\sim-\log(1-\varepsilon_0)g(\sigma_{(2n+1)})\geq cd(e^{(1+\varepsilon)(2n+1)}),\]
where we used the eventual monotonicity of $g$ guaranteed by Assumption \ref{assumpt:sosv} and the almost-sure bounds on records of Lemma \ref{lem:probrecord},
\[p_3= \frac{L(\sigma_{(2n+1)})}{L(\varepsilon_5^{-1}\sigma_{(2n+1)})}\rightarrow 1,
\]
\begin{eqnarray*}
p_4&=&\mathbf{P}\left(\frac{z_N-z_{N-1}}{\Lambda_{2n+1}}>\varepsilon_1^{-1}
\:\vline\:\sigma_{(2n+1)}\right)-
\mathbf{P}\left(\frac{z_N-z_{N-1}}{\Lambda_{2n+1}}>\varepsilon_2^{-1}\:\vline\:\sigma_{(2n+1)}\right)\\
&=&\left(1-\frac{1}{\Lambda_{2n+1}}\right)^{\Lambda_{2n+1}\varepsilon_1^{-1}}-
\left(1-\frac{1}{\Lambda_{2n+1}}\right)^{\Lambda_{2n+1}\varepsilon_2^{-1}}
\sim e^{-\varepsilon_1^{-1}}-e^{-\varepsilon_2^{-1}},
\end{eqnarray*}
which is strictly positive, and finally,
\[p_5=\mathbf{P}\left(\frac{L((1-\varepsilon_0)\varepsilon_7^{-1}S_m)}{m}<\varepsilon_6\right)\vline_{m=\varepsilon_6^{-1}\Lambda_{2n+1}}\rightarrow e^{-\varepsilon_6^{-1}}>0.\]
Hence we deduce that
\[\mathbf{P}\left( \tilde{\mathcal{E}}_{2(n+1)}  | \mathcal{F}_{2n} \right)\geq cd(e^{(1+\varepsilon)(2n+1)})^{N-2}.\]
Noting that a monotone sequence $a_n$ is summable if and only if $a_{\floor{(1+\varepsilon)(2n+1)}}$ is summable, the above sequence is not summable under Assumption \ref{assumpt:g22}, completing the proof.
\end{proof}

Note that the previous lemma shows that, under the same conditions, Assumption~\ref{assump} holds. Thus, together with Corollary \ref{cor:main2}, this completes the proofs of Theorems \ref{thm:main1} and~\ref{thm:main2}. Moreover, together with Corollary \ref{cor:main3}, it completes the proof of Theorem \ref{thm:main3}.

\bibliography{paper}{}

\begin{thebibliography}{10}

\bibitem{Aldous82}
D.~J. Aldous.
\newblock Some inequalities for reversible {M}arkov chains.
\newblock {\em J. London Math. Soc. (2)}, 25(3):564--576, 1982.

\bibitem{BenArous12}
G.~{Ben Arous} and O.~G{\"u}n.
\newblock Universality and extremal aging for dynamics of spin glasses on
  subexponential time scales.
\newblock {\em Comm. Pure Appl. Math.}, 65:77--127, 2012.

\bibitem{Bingham87}
N.~H. Bingham, C.~M. Goldie, and J.~L. Teugels.
\newblock {\em Regular Variation}.
\newblock Cambridge University Press, 1987.

\bibitem{Bovier13}
A.~Bovier, V.~Gayrard, and A.~\v{S}vejda.
\newblock Convergence to extremal processes in random environments and extremal
  ageing in {SK} models.
\newblock {\em Probab. Theory Related Fields}, 157:251--283, 2013.

\bibitem{Croydon13}
D.~Croydon, A.~Fribergh, and T.~Kumagai.
\newblock Biased random walk on critical {G}alton--{W}atson trees conditioned
  to survive.
\newblock {\em Probab. Theory Related Fields}, 157:453--507, 2013.

\bibitem{Croydon15}
D.~Croydon and S.~Muirhead.
\newblock Functional limit theorems for the {B}ouchaud trap model with slowly
  varying traps.
\newblock {\em Stoch. Process. Appl.}, 125(5):1980--2009, 2015.

\bibitem{Croydon16}
D.~Croydon and S.~Muirhead.
\newblock Quenched localisation in the {B}ouchaud trap model with regularly
  varying traps.
\newblock {\em arXiv:1601.00514}, 2016.

\bibitem{Darling52}
D.~A. Darling.
\newblock The influence of the maximum term in the addition of independent
  random variables.
\newblock {\em Trans. Amer. Math. Soc.}, 73:95--107, 1952.

\bibitem{FIN99}
L.~R.~G. Fontes, M.~Isopi, and C.~M. Newman.
\newblock Chaotic time dependence in a disordered spin system.
\newblock {\em Probab. Theory Related Fields}, 115(3):417--443, 1999.

\bibitem{Gun13}
O.~G{\"u}n.
\newblock Extremal aging for trap models.
\newblock {\em arXiv:1312.1137}, 2013.

\bibitem{Kasahara86}
Y.~Kasahara.
\newblock A limit theorem for sums of i.i.d. random variables with slowly
  varying tail probability.
\newblock {\em J. Math. Kyoto. Univ.}, 37:197--205, 1986.

\bibitem{Muirhead15}
S.~Muirhead.
\newblock Two-site localisation in the {B}ouchaud trap model with slowly
  varying traps.
\newblock {\em Electron. Commun. Probab.}, 20(25):1--15, 2015.

\bibitem{Sinai82}
Y.~G. Sina{\u{i}}.
\newblock The limit behavior of a one-dimensional random walk in a random
  environment.
\newblock {\em Teor. Veroyatnost. i Primenen.}, 27:247–258, 1982.

\bibitem{Zeitouni}
O.~Zeitouni.
\newblock Random walks in random environment.
\newblock In {\em Lectures on probability theory and statistics}, volume 1837
  of {\em Lecture Notes in Math.}, pages 189--312. Springer, Berlin, 2004.

\end{thebibliography}
\bibliographystyle{plain}

\end{document}